\pdfoutput=1
\documentclass{article}
\usepackage{ebproof}
\usepackage{cellspace}
\usepackage{amsmath}
\usepackage{amssymb}
\usepackage{thmtools}
\usepackage{thm-restate}
\usepackage{amsthm}
\usepackage{stmaryrd}
\usepackage{mathpartir}
\usepackage{mathtools}
\usepackage{quiver}

\usepackage{bbm}
\usepackage{bm}
\usepackage{xspace}
\usepackage[margin=1.6in]{geometry}
\usepackage{microtype}
\usepackage{hyperref}
\usepackage[capitalise,noabbrev]{cleveref}
\usepackage[
backend=biber,%
style=alphabetic,%
backref=false,%
useprefix=true,%
maxnames=8,%
minnames=7,%
minalphanames=3,%
maxalphanames=4,%
url=false,
eprint=true,
backrefstyle=two]%
{biblatex}
\renewcommand{\-}[0]{\nobreakdash-\hspace{0pt}}
\tikzset{dot/.style={draw=none, circle, fill=black, inner sep=1pt}}

\def\su{\textsf{su}\xspace}
\def\sua{\textsf{sua}\xspace}
\def\Catt{\textsc{Catt}\xspace}
\def\Cattsua{\textsc{Catt}\textsubscript{\sua}\xspace}
\def\Cattsu{\textsc{Catt}\textsubscript{\su}\xspace}

\allowdisplaybreaks

\newcommand\mypara[1]{

\vspace{4pt}
\noindent
\emph{#1}}

\tikzset{
  Rightarrow/.style={double equal sign distance,>={Implies},->},
  triple/.style={-,preaction={draw,Rightarrow}}
}
\tikzcdset{column sep/smaller/.initial=0em}

\DeclareMathOperator{\id}{id}
\newcommand*{\Coh}[3]{\ensuremath\mathsf{coh}\;(#1:#2)[#3]}
\newcommand*{\Ctx}{\ensuremath{\mathsf{Ctx}}}

\newcommand*{\Type}{\ensuremath{\mathsf{Type}}}
\newcommand*{\Term}{\ensuremath{\mathsf{Term}}}
\newcommand*{\arr}[3]{\ensuremath{#1 \to_{#2} #3}}
\newcommand*{\sub}[1]{\ensuremath{\llbracket #1 \rrbracket}}
\newcommand*{\supp}{\ensuremath{\mathsf{supp}}}
\newcommand*{\bound}[2]{\ensuremath{\partial^{#1}_{#2}}}
\newcommand*{\inc}[3]{\ensuremath{\delta^{#1,#2}_{#3}}}
\newcommand\doubleplus{+\kern-1.3ex+\kern0.8ex}
\newcommand*{\insertion}[3]{\ensuremath{#1\mathop{\mathord{\ll}_{#2}}#3}}
\newcommand*{\insertionprime}[3]{\ensuremath{#1\mathop{\mathord{\ll'}_{#2}}#3}}
\renewcommand*{\th}{\ensuremath{\mathsf{th}}}
\newcommand*{\bh}{\ensuremath{\mathsf{bh}}}
\newcommand*{\lh}{\ensuremath{\mathsf{lh}}}
\newcommand*{\+}{\mathbin{\#}}
\DeclareMathOperator*{\bighash}{\text{\LARGE \(\+\)}}
\renewcommand*{\sc}{\ensuremath{\mathsf{sc}}}

\newtheorem{theorem}{Theorem}
\newtheorem{prop}[theorem]{Proposition}
\newtheorem{cor}[theorem]{Corollary}
\newtheorem{lemma}[theorem]{Lemma}
\theoremstyle{definition}
\newtheorem{definition}[theorem]{Definition}
\newtheorem{example}[theorem]{Example}

\theoremstyle{remark}

\usepackage{todonotes}
\usepackage{lipsum}

\begin{document}

\title{\bf A Syntax for Strictly Associative\\and Unital \(\infty\)-Categories}

\author{Eric Finster\footnote{University of Birmingham, \texttt{e.l.finster@bham.cs.ac.uk}}~,
Alex Rice\footnote{University of Cambridge, \texttt{alex.rice@cl.cam.ac.uk}}~,
and Jamie Vicary\footnote{University of Cambridge, \texttt{jamie.vicary@cl.cam.ac.uk}}}
% \author{Eric Finster}
% \orcid{0000-0002-6027-7488}
% \affiliation{
%   \institution{University of Birmingham}
%   \city{Birmingham}
%   \country{UK}
% }
% \email{e.l.finster@bham.ac.uk}

% \author{Alex Rice}
% \orcid{0000-0002-2698-5122}
% \affiliation{
%   \institution{University of Cambridge}
%   \city{Cambridge}
%   \country{UK}
% }
% \email{alex.rice@cl.cam.ac.uk}

% \author{Jamie Vicary}
% \orcid{0000-0002-0998-1701}
% \affiliation{
%   \institution{University of Cambridge}
%   \city{Cambridge}
%   \country{UK}
% }
% \email{jamie.vicary@cl.cam.ac.uk}

\maketitle

\noindent
\begin{abstract}
  We present the first definition of strictly associative
  and unital $\infty$-category. Our proposal takes the form of a type
  theory whose terms describe the operations of such structures, and
  whose definitional equality relation enforces desired strictness
  conditions.  The key technical device is a new computation rule in the definitional equality of
  the theory, which we call \emph{insertion}, defined in terms of a universal property.  On terms for which it is
  defined, this operation ``inserts'' one of the arguments of a
  substituted coherence into the coherence itself, appropriately
  modifying the pasting diagram and result type, and simplifying the
  syntax in the process.  We generate an equational theory from this
  reduction relation and we study its properties in detail, showing
  that it yields a decision procedure for equality.

  Expressed as a type theory, our model is well-adapted for generating
  and verifying efficient proofs of higher categorical statements. We
  illustrate this via an OCaml implementation, and give a number of
  examples, including a short encoding of the syllepsis, a
  5-dimensional homotopy that plays an important role in the homotopy
  groups of spheres.
\end{abstract}

\section{Introduction}
\label{sec:introduction}

\mypara{Background.}
The theory of higher categories has growing importance in computer science, mathematics, and physics, with fundamental applications now recognized in type theory~\cite{Hofmann1994, hottbook}, quantum field theory~\cite{Atiyah2009, CSPthesis}, and geometry~\cite{Lurie2009}. Its relevance for logic was made abundantly clear by Hoffman and Streicher~\cite{Hofmann1994}, whose groupoid model of Martin-L\"of type theory violated the principle of \textit{uniqueness of identity proofs}~(UIP). This paved the way to the modern study of proof-relevant logical systems, in which one can reason about proofs themselves as fundamental objects, just as one may study the homotopy theory of paths of a topological space.

While this \textit{proofs-as-paths} perspective yields considerable additional power, a key drawback is the extensive additional proof obligations which can arise when working with the resulting path types. These are commonly understood to organize into three classes: the \emph{associator} witnessing that for three composable paths $f,g,h$, the compositions $f \circ (g \circ h)$ and $(f \circ g) \circ h$ are equivalent; the \textit{unitors} witnessing that $\id \circ f$ and $f \circ \id$ should both be equivalent to $f$ itself; and the \emph{interchanger} witnessing that compositions in different dimensions should commute. Worse, such path witnesses themselves admit further constraints, such as the \emph{pentagon condition} for associators, for which further higher-dimensional witnesses must be computed, compounding the problem exponentially in higher dimensions.

The need to construct and manage such witnesses (also called \textit{weak structure}) can complicate a path-relevant proof, potentially beyond the point of tractability. Indeed an explicit proof of the \textit{syllepsis}, an important homotopy from low-dimensional topology which is  entirely formed from such path witnesses, was formalized only recently~\cite{sojakova2022syllepsis, StrictUnits}. This motivates the search for a syntax that can trivialize witnesses as far as possible, an effort which has been underway in the mathematics community since at least the 1970s. An early realization was that in the fully strict case, where all witnesses are trivialized, too much expressivity is lost, with  Grothendieck being one of the first to observe this~\cite[page~2]{PursuingStacks}.
%was one of the first to observe this~\cite{PursuingStacks}.%, and a nice discussion of this history is given by Simpson~\cite{Simpson1998}.

The focus therefore turned to identifying a \textit{semistrict} theory, where as many witnesses as possible are trivialized, while retaining equivalence with the fully weak theory. An early contribution by Gray~\cite{Gray1974} yielded a definition of semistrict tricategory with trivial  associators and unitors, leaving only the interchangers nontrivial; later work by Gordon, Power and Street showed that this definition loses no expressibility~\cite{Gordon1995}. Simpson conjectured that ``it suffices to weaken any one of the principal structures involved'' to obtain a definition of semistrict $n$\-category~\cite{Simpson1998}, and Street had  sketched a possible approach based on iterated enrichment, where a closed monoidal category of semistrict $n$\-categories is defined in each dimension~\cite{Street1995}. However, a subsequent analysis by Dolan~\cite{Dolan1996} indicated that the required monoidal closure properties could not be fulfilled.  One can see from this episode that even finding a reasonable \emph{definition} of these objects is a non-trivial  task.

\begin{figure}
\begin{align*}
\text{{(a)}}
&
\quad
T:=\!\!\!
\begin{aligned}
\begin{tikzpicture}[xscale=.3, yscale=.6, scale=.8, font=\small]
\draw (0,0) node [dot] {} to (-1,1) node [dot] {} to (-2,2) node [dot] {} node [above] {$\alpha\vphantom{\beta}$};
\draw (-1,1) to (0,2) node[dot]{} node [above] {$\beta$};
\draw (0,0) to (1,1) node[dot]{} node [above] {$\gamma$};
\end{tikzpicture}
\end{aligned}
\,\,,\,\,
\text{arg}(\alpha) := \!\!\!
\begin{aligned}
\begin{tikzpicture}[xscale=.3, yscale=.6, scale=.8, font=\small, red]
\draw (-1,0) node [dot, red] {} to (-1,1) node [dot, fill=red] {} to (-2,2) node [dot, fill=red] {} node [above] {$\phi\vphantom{\beta}$};
\draw (-1,1) to (0,2) node[dot, red]{} node [above] {$\psi$};
\end{tikzpicture}
\end{aligned}
\,\,\leadsto\,\,
T':=\!\!\!
\begin{aligned}
\begin{tikzpicture}[xscale=.4, yscale=.6, scale=.8, font=\small]
\draw (-1,1) to (0,2) node [dot] {} node [above] {$\beta$};
\draw [red] (0,0) node [dot, red] {} to (-1,1) node [dot, red] {} to (-2,2) node [dot, red] {} node [above] {$\phi\vphantom{\beta}$};
\draw [red] (-1,1) to (-1,2) node[dot, red]{} node [above] {$\psi$};
\draw (0,0) to (1,1) node[dot]{} node [above] {$\gamma$};
\end{tikzpicture}
\end{aligned}
\\
\text{{(b)}}
&
\quad
T:=\!\!\!
\begin{aligned}
\begin{tikzpicture}[xscale=.3, yscale=.6, scale=.8, font=\small]
\draw (0,0) node [dot] {} to (-1,1) node [dot] {} to (-2,2) node [dot] {} node [above] {$\alpha\vphantom{\beta}$};
\draw (-1,1) to (0,2) node[dot]{} node [above] {$\beta$};
\draw (0,0) to (1,1) node[dot]{} node [above] {$\gamma$};
\end{tikzpicture}
\end{aligned}
\,\,,\,\,
\text{arg}(\alpha) := \!\!\!
\begin{aligned}
\begin{tikzpicture}[xscale=.3, yscale=.6, scale=.8, font=\small, red]
\draw (-1,0) node [dot, red] {} to (-1,1) node [dot, fill=red] {} node [above] {$\phi$};
\node [dot, white] at (-1,2) {};
\node [above, white] at (-1,2) {$\phi$};
\node [white] at (0,0) {$\phi$};
\node [white] at (-2,0) {$\phi$};
\end{tikzpicture}
\end{aligned}
\,\,\leadsto\,\,
T':=\!\!\!
\begin{aligned}
\begin{tikzpicture}[xscale=.4, yscale=.6, scale=.8, font=\small]
\draw (-1,1) to (-1,2) node [dot] {} node [above] {$\beta$};
\draw [red] (0,0) node [dot, red] {} to (-1,1) node [dot, red] {};
\draw (0,0) to (1,1) node [dot] {} node [above] {$\gamma$};
\end{tikzpicture}
\end{aligned}
\\
\text{{(c)}}
&
\quad
T:=\!\!\!
\begin{aligned}
\begin{tikzpicture}[xscale=.3, yscale=.6, scale=.8, font=\small]
\draw (0,0) node [dot] {} to (-1,1) node [dot] {} to (-2,2) node [dot] {} node [above] {$\alpha\vphantom{\beta}$};
\draw (-1,1) to (0,2) node[dot]{} node [above] {$\beta$};
\draw (0,0) to (1,1) node[dot]{} node [above] {$\gamma$};
\end{tikzpicture}
\end{aligned}
\,\,,\,\,
\text{arg}(\alpha) := \!\!\!
\begin{aligned}
\begin{tikzpicture}[xscale=.3, yscale=.6, scale=.8, font=\small, red]
\draw (0,0) node [dot, red] {} to (-1,1) node [dot, fill=red] {} node [above] {$\phi$};
\node [dot, white] at (-1,2) {};
\node [above, white] at (-1,2) {$\phi$};
\draw (0,0) to (1,1) node [dot, red] {} node [above] {$\psi$};
\end{tikzpicture}
\end{aligned}
\,\,\hspace{2pt} \mathrlap{\hspace{2pt}/}{\leadsto}\,\,
\end{align*}
\caption{Illustrating the insertion operation.}
\label{fig:insertionintro}
\end{figure}

\mypara{Contribution.}
We define $\infty$\-categories with strict associators and unitors, so that only the interchanger structure remains weak, by describing a type theory whose terms express operations valid in such a structure.  To our knowledge this is the first such definition to have been presented in the literature. In this theory the operations give the compositional structure and coherence laws of the $\infty$\-category, while the equational theory trivializes the associator and unitor structure. For example, the equation $(f \circ g) \circ h = f \circ (g \circ h)$ is derivable in our theory.

Our proposal is implemented as a type theory, and the type checker can apply such equalities automatically, without requiring the user to specify explicit associator witnesses. This allows the user to work directly with the semistrict theory, building proofs which  neglect associator and unitor witnesses and their higher-dimensional counterparts, with the computer nonetheless able to verify correctness of the resulting terms. For example, this would enable the user to directly compose higher paths of type $p \Rightarrow (f \circ g) \circ h$ and $f \circ (g \circ h) \Rightarrow q$, even though it seems they should not be directly composable.

Alternatively, given a proof which includes all such explicit witnesses, the system can ``compute them out'', returning a normal form where associator and unitor structure has been eliminated as far as possible. This has the potential to yield a simpler proof.

Our key technical contribution is an \emph{insertion} procedure which generates the semistrict behaviour of our equality relation, and which we show satisfies a simple universal property. Any operation in our theory is defined with respect to a \textit{pasting context}, a gluing of discs that governs the composition; combinatorially, these correspond to finite planar rooted trees. Any operation then has a head tree, with arguments corresponding to the leaves of the tree. We require  two simple definitions: for a rooted tree,  its \emph{trunk height} is the height of the ``tree trunk''; and for any leaf, its \emph{branch height} is the height of the highest branch point sitting below it. See \cref{fig:leafheight} for a graphical illustration of these definitions.  For a compound operation, where arguments of the head tree are themselves assigned further operations, insertion operates as follows, for some chosen leaf $\alpha$ of the head tree: if the branch height of $\alpha$ is at most the trunk height of the argument tree at $\alpha$, and if the argument operation is in ``standard form'', we can insert the argument tree into the head tree.

We illustrate this in \cref{fig:insertionintro}, using Greek letters to label the leaves. In each part, the head tree $T$ has an argument $\alpha$ of branch height 1 (since descending from $\alpha$, the first branch point we encounter is at height 1.) In part (a) the argument tree drawn in red has trunk height 1, and so the height test is satisfied and insertion can proceed, yielding an updated head tree $T'$ in which  the entire argument tree has been inserted and is now visible in red as a subtree, replacing the original leaf~$\alpha$. In part (b) the argument tree also has trunk height 1, and so the insertion can again proceed, although in this case, due to the short stature of the argument tree, the effect is simply to remove the leaf $\alpha$ of the head tree. In part (c) the trunk height of the argument tree is 0, so the height test is not satisfied and insertion cannot proceed. This last example makes it intuitively clear why the height test is important: since the trunk height of the argument tree is so small, there would be no reasonable way to insert the argument tree into the head tree without disrupting the leaf~$\beta$. We show in the paper that this insertion operation can be described algebraically via a universal property.

We combine insertion with two other simpler procedures, \emph{disc removal} and \emph{endo-coherence removal}, to produce our reduction relation on operations of the theory. Our major technical results concern the behaviour of this reduction relation. We show that it terminates, and that it has unique normal forms, yielding a decision procedure for equality.

\mypara{Implementation.}
We have produced an OCaml implementation of our scheme as a type theory, which we call $\Cattsua$, standing for ``categorical type theory with strict unitors and associators''. In \cref{sec:implementation} we give three worked examples, showing that the triangle and pentagon conditions of the definition of monoidal category completely trivialize, and demonstrating that a known proof of the 5\-dimensional syllepsis homotopy reduces to a substantially simpler form. The implementation can be found at:
\begin{center}
  \url{https://github.com/ericfinster/catt.io/tree/lics-sua}
\end{center}

\mypara{Formalisation.}
The majority of technical results in this paper have been formalised in Agda, and we make this available here~\cite{alex_rice_2024_11149192}:
\begin{center}
  \url{https://github.com/alexarice/catt-agda/tree/paper}
\end{center}

\noindent
The formalisation compiles with Agda version 2.6.4.1 and standard library version 2.0. Where a result has been formalised, we supply the file and line number of the corresponding formal proof. In the  Agda files, a comment above a formalised proof indicates the corresponding theorem number here.

\mypara{Related Work.} Our work is based on the theory of \textit{contractible $\infty$\-categories}, a well-studied model of globular $\infty$\-categories originally due to Maltsiniotis~\cite{maltsiniotis2010grothendieck}, who was building on an early algebraic definition of $\infty$\-groupoid by Grothendieck~\cite{PursuingStacks}. An excellent modern presentation is given by Leinster in terms of contractible globular operads~\cite{leinster2004higher}.

Our algebraic theory builds on an existing type theory \Catt for contractible $\infty$\-categories presented at LICS 2017~\cite{Finster2017}, and an extension $\Cattsu$ presented at LICS 2022~\cite{StrictUnits} which describes the strictly unital case. That work includes a reduction relation called \emph{pruning}, which removes a single leaf variable from a pasting context. Here the pruning operation is replaced  by insertion, which includes pruning as a special case, but which in general performs far more radical surgery on the head context. As a result, the termination and confluence properties here are significantly more complex to establish. Our termination proof uses new techniques that  quantify the syntactic complexity of a term, while our confluence proof must analyze many additional critical pairs, some of which are  fundamentally more complex than those handled previously. To allow surgery on pasting contexts we also require a  different presentation of contexts in terms of trees, which changes many aspects of the formal development. Furthermore, unlike the LICS 2022 paper, all technical lemmas here are formalized, putting the work on a stronger foundation.

Above we discussed the problem faced by Street's historical approach to defining semistrict $n$\-categories. His approach aimed to build up the semistrict theory one dimension at a time; in contrast, we retain all the operations of the fully weak structure, instead obtaining  semistrictness via a reduction relation on operations. We do not require the explicit definition of a tensor product on semistrict $n$\-categories, nor do we require any closure properties. As a result Dolan's critique does not apply.

%\mypara{Supplementary material.} Attached to this submission are are 2 folders containing the source for our Ocaml implementation and our Agda formalisation. The Ocaml implementation has installation instructions included.

%\emph{Some unused words we had liked: In this paper we present a new type-theoretic definition for \(\infty\)-categories which have strict unitors and associators, but weak interchangers ... Genrealiseed algebraic theories (GATs) are a crucial semantic model for presenting dependently typed theories ... Pushing GATs in a new dimension -- higher category theory}

\section{The Type Theory \Catt}
\label{sec:catt}

\noindent
In this section we recall the type theory \Catt and some of its basic properties. We generalize the original presentation~\cite{Finster2017} by parameterizing the theory over a given set of equality rules, thus allowing us to prove some general structural properties generically, and allowing future investigations into strictness results to build on the theory developed here.  We will then specialize to the theory \Cattsua, which is the focus of the present work.

Special cases of this general framework include the original presentation of \Catt by Finster and Mimram~\cite{Finster2017} in which the set of equality rules is empty, as well as the theory \Cattsu~\cite{StrictUnits} whose rules we recall below in the present framework.

\subsection{Syntax for \Catt}
\label{sec:cattbase}

% The type theory has been explained and motivated in great detail in previous papers, in particular the original paper~\cite{Finster2017} by Finster and Mimram and in Benjamin's thesis~\cite{benjamin:tel-03106197} so that our exposition here will be brief. We draw attention however to the structure of types, which give the theory its globular flavour, and can be compared to identity types in MLTT/HoTT\todo{EF: is abbreviation fine here? Do we really need this paragraph?}.

\noindent
\Catt has 4 classes of syntax: contexts, substitutions, types, and terms; the rules for each can be found in \cref{fig:syntax}. We parameterise substitutions, types and terms by their context in order to avoid issues with undefined variables and write \(\Term_\Gamma\) for a term in context \(\Gamma\), \(\Type_\Gamma\) for a type in context \(\Gamma\), and \(\sigma : \Gamma \to \Delta\) for a substitution from \(\Gamma\) to \(\Delta\).

We let \(\equiv\) denote syntactic equality up to alpha renaming. Our presentation will be in terms of named variables to improve readability, though in practice any ambiguity introduced by this choice can be avoided by the use of de Bruijn indices.

\begin{figure}
  \centering
    \begin{tabular}{Sc Sc}
    {
    \begin{prooftree}
      \hypo{\phantom{\Term}} \infer1{\emptyset : \Ctx}
    \end{prooftree}
    }
    &
      {
      \begin{prooftree}
        \hypo{\Gamma : \Ctx} \hypo{A : \Type_\Gamma}
        \infer2{\Gamma, (x : A) : \Ctx}
      \end{prooftree}}
    \\
    {
    \begin{prooftree}
      \hypo{\phantom{\Term}} \infer1{\langle \rangle : \emptyset \to \Gamma}
    \end{prooftree}
    }
    & {
      \begin{prooftree}
        \hypo{\sigma : \Delta \to \Gamma} \hypo{t : \Term_\Gamma}
        \infer2{\langle \sigma , x \mapsto t \rangle : (\Delta, x : A) \to \Gamma}
      \end{prooftree}
      }
    \\
    {
    \begin{prooftree}
      \hypo{\phantom{\Type}} \infer1{\star : \Type_\Gamma}
    \end{prooftree}
    }
    & {
      \begin{prooftree}
        \hypo{A : \Type_\Gamma} \hypo{s : \Term_\Gamma} \hypo{t : \Term_\Gamma}
        \infer3{\arr s A t : \Type_\Gamma}
      \end{prooftree}
      }
    \\
    {
    \begin{prooftree}
      \hypo{x \in \Gamma\vphantom{\Type}} \infer1{x : \Term_\Gamma}
    \end{prooftree}
    }
    & {
      \begin{prooftree}
        \hypo{\Delta : \Ctx} \hypo{A : \Type_\Delta} \hypo{\sigma : \Delta \to \Gamma}
        \infer3{\Coh \Delta A \sigma : \Term_\Gamma}
      \end{prooftree}
      }

    \end{tabular}
  \caption{Syntax constructions in \Catt.}
  \label{fig:syntax}
\end{figure}

A substitution \(\sigma : \Delta \to \Gamma\) maps variables of context \(\Delta\) to terms of context \(\Gamma\). We note that this is non-standard within type theory, but agrees with the direction of morphisms in the coherators of Grothendieck and Maltsiniotis~\cite{maltsiniotis2010grothendieck} on which \Catt is based. For any \(t : \Term_\Delta\) , \(A : \Type_\Delta\), and \(\tau : \Theta \to \Delta\), one defines semantic substitution operations
\[t \sub \sigma : \Term_\Gamma \qquad A \sub \sigma : \Type_\Gamma \qquad \tau \circ \sigma : \Theta \to \Gamma \]
by mutual recursion on types, terms and substitutions (as in \cite{Finster2017}). Note the use of doubled brackets $\sub -$ to denote this operation, which we are careful to distinguish from the single brackets $[-]$ which are part of the syntax of the coherence constructor. Every context has an identity substitution \(\id_{\Gamma} : \Gamma \to \Gamma\) which maps each variable to itself and the operation of substitution is associative and unital so that the collection of contexts and substitutions forms a category which we will denote $\Catt$ as an abuse of notation.

% \begin{prop}
%   \label{prop:sub-assoc-unit}
%   Contexts and substitutions form a category.
% \end{prop}
% \begin{proof}

% For example, for a term \(s : \Term_\Theta\), and substitutions \(\tau : \Theta \to \Delta\),  \(\sigma : \Delta \to \Gamma\), we have $t \sub \tau \sub \sigma \equiv t \sub {\tau \circ \sigma}$,
%   and similarly for other pieces of syntax.
% \end{proof}

The set of free variables for each syntactic class is defined in a standard way by induction. Given a context \(\Gamma\) and set of variables \(S \subseteq \Gamma\), we say that \(S\) is \emph{downwards closed} when for all \(x : A \in \Gamma\), \(x \in S\) implies that the free variables of \(A\) are a subset of \(S\). For any set \(S \subseteq \Gamma\) we can form its \emph{downward closure}. We can then define the support of a piece of syntax, which intuitively is the set of variables the syntax depends on.

\begin{definition}
  Given a term \(t : \Term_\Gamma\), the \emph{support} of \(t\), \(\supp(t)\), is the downwards closure of the free variables of \(t\) in \(\Gamma\). The support of a type or substitution is defined similarly.
\end{definition}

\begin{example}
  Consider the context $\Gamma = x : \star, y : \star, f : \arr x \star y$.  Then the variable $f$ is a valid term in this context whose set of free variables is simply the singleton set $\{f\}$.  The support of $f$, however, is $\{x,y,f\}$.
\end{example}

Lastly we define the dimension of a type \(A\), \(\dim(A)\), by \(\dim(\star) = 0\) and \(\dim(\arr s A t) = 1 + \dim(A)\). The dimension of a term \(t : \Term_\Gamma\) is given by \(\dim(x) = \dim(A)\) when \(x : A \in \Gamma\) and \(\dim(\Coh \Delta A \sigma) = \dim(A)\). The dimension of a context \(\dim(\Gamma)\) is the maximum of the dimension of the types it contains.  One proves easily by induction that dimension is preserved by substitution.

\subsection{Typing for \Catt}
\label{sec:catt-typing}

\noindent
The coherences of \Catt are determined by a special class of contexts which we refer to as \emph{pasting contexts}.  These correspond intuitively to configurations of globular cells which should admit a unique composition. In particular, for a coherence term \(\Coh \Delta A \sigma\) to be well typed it is necessary that \(\Delta\) is a pasting context. Determining whether an arbitrary context is a pasting context is decidable and we write the judgment
\[ \Delta \vdash_{p}\]
when \(\Delta\) is a pasting context.

Crucial for the typing rules of \Catt is the fact that every pasting context \(\Delta\) has a well defined \emph{boundary}, which is again a pasting context. For a natural number \(n\), we can construct \(\bound n \Delta\), the \(n\)-dimensional boundary of \(\Delta\), and there are inclusion substitutions \(\inc \epsilon n \Delta : \bound n \Delta \to \Delta\) for each \(n\) and \(\epsilon \in \{-,+\}\).

Rules for pasting contexts as well as a definition of the boundary operators can be found in \cite{Finster2017}. We will give an alternative description of pasting contexts as trees in later in \cref{sec:trees}.

% \begin{alignat*}{2}
%   &\Gamma \vdash &\qquad&\text{Context \(\Gamma\) is well typed.}\\
%   &\Gamma \vdash A &&\text{Type \(A\) is well typed in \(\Gamma\).}\\
%   &\Gamma \vdash s : A && \text{Term \(s\) has type \(A\) in \(\Gamma\).}\\
%   &\Gamma \vdash \sigma : \Delta &&\text{\(\sigma\) is a valid substitution from \(\Delta\) to \(\Gamma\).}
% \end{alignat*}
\begin{figure}
  \centering
  \begin{tabular}{ScSc}
    {
    \begin{prooftree}
      \hypo{\vphantom{\Gamma \vdash A}} \infer1{\emptyset
        \vdash}
    \end{prooftree}
    }
    & {
      \begin{prooftree}
        \hypo{\Gamma \vdash}
        \hypo{\Gamma \vdash A} \infer2{\Gamma, (x : A) \vdash}
      \end{prooftree}
      }
    \\
    {
    \begin{prooftree}
      \hypo{\vphantom{\Gamma \vdash A}} \infer1{\Gamma \vdash \star}
    \end{prooftree}
    }
    & {
      \begin{prooftree}
        \hypo{\Gamma \vdash A} \hypo{\Gamma \vdash s : A} \hypo{\Gamma
          \vdash t : A} \infer3{\Gamma \vdash \arr s A t}
      \end{prooftree}
      }
    \\
    {
    \begin{prooftree}
      \hypo{\vphantom{\Delta v \vdash \sigma : \Gamma}}
      \infer1{\Delta \vdash \langle \rangle : \emptyset}
    \end{prooftree}
    }
    & {
      \begin{prooftree}
        \hypo{\Delta \vdash \sigma : \Gamma} \hypo{\Gamma \vdash A}
        \hypo{\Delta \vdash t : A \sub \sigma}  \infer3{\Delta \vdash \langle \sigma, x \mapsto
          t \rangle : (\Gamma, x : A)}
      \end{prooftree}
      }\\[1.5em]
    \multicolumn{2}{c} {
    \inferrule{\Gamma \vdash \\ (x : A) \in \Gamma}{\Gamma \vdash x : A}
    }\\[1.5em]
    \multicolumn{2}{c} {
    \inferrule{
      \Delta \vdash_{p} \\
      \Delta \vdash \arr s A t \\
      \Gamma \vdash \sigma : \Delta \\\\
      \supp(s) = \supp(\inc - {\dim(\Delta) - 1} \Delta) \\
      \supp(t) = \supp(\inc + {\dim(\Delta) - 1} \Delta)
    }{
      \Gamma \vdash \Coh \Delta {\arr s A t} \sigma : \arr {s\sub{\sigma}} {A\sub{\sigma}} {t\sub{\sigma}}
    }
    } \\[2em]
    \multicolumn{2}{c} {
    \inferrule{
      \Delta \vdash_p \\
      \Delta \vdash \arr s A t \\
      \Gamma \vdash \sigma : \Delta \\\\
      \supp(s) = \supp(t) = \Delta
    }{
      \Gamma \vdash \Coh \Delta {\arr s A t} {\sigma} : \arr {s\sub{\sigma}} {A\sub{\sigma}} {t\sub{\sigma}}
    }
    }
  \end{tabular}
  \caption{Typing rules for \Catt.}
  \label{fig:typ-rules}
\end{figure}%
With these notions in hand, the typing rules for \Catt are given in \cref{fig:typ-rules}.  All these rules are relatively standard for dependent type theories with the exception of the last two which describe the typing of coherence terms. Both of these latter rules for typing a coherence \(\Coh \Delta {\arr s A t} \sigma\) require that \(\Delta\) must be a pasting context and that \(\arr s A t\) and \(\sigma\) are both well typed. The first rule says that \(s\) must be supported by the source of \(\Delta\) and \(t\) by the target of \(\Delta\). Terms typed with this rule represent compositions. The second rule instead states that \(s\) and \(t\) are \emph{full}, they are supported by the whole context. These terms witness that any two full terms over a pasting diagram should be equivalent. We note that although having two typing rules for a constructor is non-standard, these two rules only differ on their syntactic side condition, and could be combined into a single rule. The two rules are left separate to conform to previous presentations of \Catt.

\subsection{Constructions and Examples}

\mypara{Disc Contexts.}
Among the pasting contexts, we may distinguish the \emph{disc contexts} which play an important role in further constructions. An example of a disc context can be seen in \cref{fig:d3}.

\begin{definition}
  The \(n\)-dimensional \emph{disc context} \(D^n\) has a single \(n\)-dimensional variable $d_n$, and variables \(d_k^-, d_k^+\) for each \(k < n\). We have $d ^\pm _0: \star$, and  \(d_k^\pm : d_{k-1}^- \to d_{k-1}^+\) for all $0 < k \leq n$.
\end{definition}

\noindent
Substitutions out of $D^n$ are special in that they are fully determined by a type of dimension $n$, and a term of that type. That is, given \(A : \Type_\Gamma\) and \(t : \Term_\Gamma\), there is a substitution \(\{A,t\} : D^{\dim(A)} \to \Gamma\), and any substitution from a disc is of this form.

\mypara{Standard Operations.} A given pasting context $\Gamma$ generally gives rise to many different valid coherences.  Among these, there is a particular class of coherences, the \emph{standard} ones, which will play an important role in our definitional equality below.  Intuitively speaking, these are the coherences which compose all of the cells in a pasting diagram ``at once'' instead of first composing sub-diagrams.  They may be defined recursively as follows:

\begin{definition}
  Given a pasting diagram \(\Delta\), we mutually define for all \(n\) the \emph{standard coherence} \(\mathcal{C}_\Delta^n\), the \emph{standard term} \(\mathcal{T}_\Delta^n\), and the \emph{standard type} \(\mathcal{U}_\Delta^n\):
  \begin{align*}
    \mathcal{C}_\Delta^n &= \Coh \Delta {\mathcal{U}_\Delta^n} {\id}\\
    \mathcal{T}_\Delta^n &=
                           \begin{cases}
                             d_n &\text{when \(\Delta\) is a disc}\\
                             \mathcal{C}_\Delta^n &\text{otherwise}
                           \end{cases}\\
    \mathcal{U}_\Delta^0 &= \star\\
    \mathcal{U}_\Delta^{n+1} &= \arr {\mathcal{T}_{\bound n \Delta}^n \sub {\inc - n \Delta}} {\mathcal{U}_\Delta^n} {\mathcal{T}_{\bound n \Delta}^n \sub {\inc + n \Delta}}
  \end{align*}
  The standard type takes the standard term over each boundary of \(\Delta\), includes these all back into \(\Delta\) and assembles them into a type. When \(n = \dim(\Delta)\) we will refer to the standard coherence as an \emph{standard composite}. % TODO might not be needed
\end{definition}

\mypara{First Examples.} We start with some basic examples of categorical operations. The following pasting context contains the composable pair of morphisms $f:x \to y$ and $g:y \to z$:
\[ \Gamma = x : \star, y : \star, f : \arr x \star y, z : \star, g : \arr y \star z\]
We can use this to form the composite of \(f\) and \(g\), as a term in $\Gamma$:
\[ f \cdot g := \Coh \Gamma {\arr x \star z} \id : \arr x \star z\]
This satisfies the support conditions for the first typing rule for coherences, since $x$ is full over the context $x : \star$, and similarly for $z$.
Given a variable \(t\) of type \(\star\), we can form the identity coherence associated to \(t\) (not to be confused with the identity substitution) as follows:
\[ \mathbbm{1}_t := \Coh {(x : \star)} {\arr x \star x} {\langle x \mapsto t \rangle} : \arr t \star t \]
This can be typed using the second typing rule for coherences.

\begin{figure}
  \centering
  % https://q.uiver.app/?q=WzAsMixbMCwwLCJkXzBeLSJdLFsyLDAsImRfMF4rIl0sWzAsMSwiZF8xXisiLDAseyJjdXJ2ZSI6LTV9XSxbMCwxLCJkXzFeLSIsMix7ImN1cnZlIjo1fV0sWzMsMiwiZF8yXi0iLDAseyJvZmZzZXQiOi01LCJzaG9ydGVuIjp7InNvdXJjZSI6MjAsInRhcmdldCI6MjB9fV0sWzMsMiwiZF8yXisiLDIseyJvZmZzZXQiOjUsInNob3J0ZW4iOnsic291cmNlIjoyMCwidGFyZ2V0IjoyMH19XSxbNCw1LCJkXzMiLDAseyJzaG9ydGVuIjp7InNvdXJjZSI6MjAsInRhcmdldCI6MjB9fV1d
\[\begin{tikzcd}
        {d_0^-} && {d_0^+}
        \arrow[""{name=0, anchor=center, inner sep=0}, "{d_1^+}", curve={height=-30pt}, from=1-1, to=1-3]
        \arrow[""{name=1, anchor=center, inner sep=0}, "{d_1^-}"', curve={height=30pt}, from=1-1, to=1-3]
        \arrow[""{name=2, anchor=center, inner sep=0}, "{d_2^-}", shift left=5, shorten <=8pt, shorten >=8pt, Rightarrow, from=1, to=0]
        \arrow[""{name=3, anchor=center, inner sep=0}, "{d_2^+}"', shift right=5, shorten <=8pt, shorten >=8pt, Rightarrow, from=1, to=0]
        \arrow["{d_3}", shorten <=4pt, shorten >=4pt, triple, from=2, to=3]
\end{tikzcd}\]
  \caption{The disc context \(D^3\).}
  \label{fig:d3}
\end{figure}

The substitution part of the coherence allows us to form compound operations. For example, the following syntax represents a term in the context \(\Delta = x : \star, y : \star, f : \arr x \star y\):
\[ f \cdot \mathbbm{1}_y := \Coh \Gamma {\arr x \star z} {\langle f \mapsto f, g \mapsto \mathbbm{1}_y \rangle} : \arr x \star y \]
Here we omit the lower-dimensional components of the substitution as they can be inferred, a technique that we will use repeatedly. Building on this last example we can form the unitor \(\rho_f\) witnessing the unitality of \(\mathbbm{1}\):
\[ \rho_f := \Coh \Delta {\arr {f \cdot \mathbbm{1}_y} {\arr x \star y} {f}} \id : \arr {f \cdot \mathbbm{1}_y} {\arr x \star y} {f}  \]
and similarly we can generate the associator as a coherence over the context \(\Theta = \Gamma , w : \star, h : \arr z \star w\).
\begin{align*}
  \alpha_{f,g,h} := \Coh \Theta {{}&\arr {(f \cdot g) \cdot h} {\arr x \star w} {f \cdot (g \cdot h)}} \id \\
  &: \arr {(f \cdot g) \cdot h} {\arr x \star w} {f \cdot (g \cdot h)}
\end{align*}

We note that both the composition and identity examples are in fact examples of standard coherences, as can be seen by straightforward calculation.

More generally we can define the \textit{identity term} for any \(n\)\-dimensional term \(s\) of type \(A\) as follows:
\[ \mathbbm{1}\sub{\{A,s\}} = \mathcal{C}_{D^n}^{n+1}\sub{\{A,s\}} \]

\subsection{Trees}
\label{sec:trees}

\noindent
Pasting diagrams like the ones used in \Catt have a well known correspondence to finite planar rooted trees (henceforth, simply ``trees'') \cite{street2000petit}. Our insertion construction (see \cref{sec:insertion}) is more easily defined using the representation of a pasting context as a tree, and so we pause to work out this correspondence here.

\mypara{Suspension.}
In topology, given a space \(X\), we can construct a new space \(\Sigma X\) which is obtained by stretching \(X\) into a cylinder and then collapsing each of the top and bottom ``caps'' to a point. More formally, the suspension is a quotient of the product space \(X \times [0,1]\), where everything in \(X \times \{0\}\) is identified, and everything in \(X \times \{1\}\) is identified. This construction may also be described simply in terms of paths: the space \(\Sigma X\) has two points with a path between these points for each element of \(X\). As an example, suspending a circle yields a sphere with the original circle embedded as the equator.  Each point of the circle then determines a meridian, and this motivates the convention of calling the two additional points ``North'' and ``South''.

An analogue of the suspension operation exists in the theory \Catt~\cite{benjamin2019suspension}. As with many constructions in \Catt, it is mutually inductively defined on all pieces of syntax.
\begin{itemize}
\item For context \(\Gamma\), its suspension \(\Sigma\Gamma\) has two new variables \(N : \star, S : \star\) (\(N\) for north and \(S\) for south), as well as a variable \(x : \Sigma(A)\) for each \(x : A \in \Gamma\).
\item For type \(A : \Type_\Gamma\), its suspension \(\Sigma A : \Type_{\Sigma(\Gamma)}\) is given by \(\Sigma \star = \arr N \star S\) and \(\Sigma(\arr s A t) = \arr {\Sigma s} {\Sigma A} {\Sigma t}\). Note that this raises the dimension of the type by 1.
\item For term \(s : \Term_\Gamma\), its suspension \(\Sigma s : \Term_{\Sigma(\Gamma)}\) is defined by \(\Sigma x = x\) for variables \(x \in \Gamma\), and \(\Sigma(\Coh \Delta A \sigma) = \Coh {\Sigma\Delta} {\Sigma A} {\Sigma \sigma} \).
\item For substitution \(\sigma : \Delta \to \Gamma\), the suspension \(\Sigma \sigma : \Sigma \Delta \to \Sigma \Gamma\) sends \(N\) to \(N\), \(S\) to \(S\) and \(x\) to \(\Sigma t\) for each \(x \mapsto t \in \sigma\).
\end{itemize}
We note that for \(\Sigma(\Coh \Delta A \sigma)\) to be well typed we must have that \(\Sigma \Delta\) is a pasting diagram. This is in fact the case whenever \(\Delta\) is itself a pasting diagram. One can additionally show that suspension forms a functor $\Sigma : \Catt \to \Catt$ on the category of contexts.

\begin{example}
  The suspension of \(x: \star, y : \star, f : \arr x \star y, z : \star, g : \arr y \star z\), is the following context:
  % https://q.uiver.app/?q=WzAsMixbMCwwLCJOIl0sWzIsMCwiUyJdLFswLDEsInoiLDAseyJjdXJ2ZSI6LTV9XSxbMCwxLCJ4IiwyLHsiY3VydmUiOjV9XSxbMCwxLCJ5IiwxXSxbNCwyLCJnIiwwLHsic2hvcnRlbiI6eyJzb3VyY2UiOjIwLCJ0YXJnZXQiOjIwfX1dLFszLDQsImYiLDAseyJzaG9ydGVuIjp7InNvdXJjZSI6MjAsInRhcmdldCI6MjB9fV1d
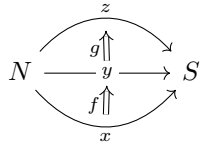
\[\begin{tikzcd}
    N && S
    \arrow[""{name=0, anchor=center, inner sep=0}, "z", curve={height=-25pt}, from=1-1, to=1-3]
    \arrow[""{name=1, anchor=center, inner sep=0}, "x"', curve={height=25pt}, from=1-1, to=1-3]
    \arrow[""{name=2, anchor=center, inner sep=0}, "y"{description}, from=1-1, to=1-3]
    \arrow["g", shorten <=4pt, shorten >=4pt, Rightarrow, from=2, to=0, pos=.4]
    \arrow["f", shorten <=4pt, shorten >=4pt, Rightarrow, from=1, to=2, pos=.4]
  \end{tikzcd}
\]
and the suspension of the 1-composition operation gives the vertical composition of 2 cells.
\end{example}

\begin{lemma}
  \label{lem:susp-standard}
  For all \(n\), \(\Sigma D^n \equiv D^{n+1}\). Furthermore, for all pasting contexts \(\Delta\):
  \[ \Sigma\mathcal{C}_\Delta^n \equiv \mathcal{C}_{\Sigma\Delta}^{n+1} \quad \Sigma\mathcal{T}_\Delta^n \equiv \mathcal{T}_{\Sigma\Delta}^{n+1} \quad \Sigma\mathcal{U}_\Delta^n \equiv \mathcal{U}_{\Sigma\Delta}^{n+1} \]
\end{lemma}

\mypara{Wedge Sum.}
In topology, the wedge sum of two pointed spaces \((X, x)\), \((Y , y)\) is the disjoint union of \(X \coprod Y\), with \(x\) identified with \(y\). This may be realized as a colimit of the following diagram:
% https://q.uiver.app/?q=WzAsMyxbMCwwLCJYIl0sWzIsMCwiWSJdLFsxLDEsIlxcYnVsbGV0Il0sWzIsMCwieCJdLFsyLDEsInkiLDJdXQ==
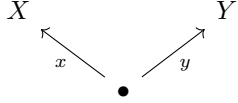
\[\begin{tikzcd}
    X && Y \\
    & \bullet
    \arrow["x", from=2-2, to=1-1]
    \arrow["y"', from=2-2, to=1-3]
  \end{tikzcd}
\]
A similar construction may be made for pasting contexts.  We note first that it is a basic consequence of the their definition that every pasting context begins with a variable of type $\star$.  Now, given $\Gamma = \Gamma', (x : \star), \Gamma''$ where $\Gamma''$ contains no variables of type $\star$ (i.e., the variable $x$ is the \emph{last} variable of this type) and $\Delta = (y : \star) , \Delta'$, we may define
\[ \Gamma \vee \Delta := \Gamma', (x : \star) , \Gamma'', \Delta'[x/y] \]
where $\Delta'[x/y]$ denotes the result of substituting $x$ for $y$ in all the types which appear in $\Delta'$.  One easily checks that the result is again a pasting diagram.  Moreover, this construction
has an obvious extension to multiple pasting contexts, which we write as:
\[ \bigvee\limits_{i = 1}^n \Gamma_i = \Gamma_1 \vee \Gamma_2 \vee \dots \vee \Gamma_n\]%
\begin{example}
  We consider \(D^2 \vee D^2\). This glues two 2-discs at a point and yields a context of the following form:
  % https://q.uiver.app/?q=WzAsMyxbMCwwLCJcXGJ1bGxldCJdLFsxLDAsIlxcYnVsbGV0Il0sWzIsMCwiXFxidWxsZXQiXSxbMCwxLCIiLDAseyJjdXJ2ZSI6LTJ9XSxbMCwxLCIiLDIseyJjdXJ2ZSI6Mn1dLFsxLDIsIiIsMix7ImN1cnZlIjotMn1dLFsxLDIsIiIsMix7ImN1cnZlIjoyfV0sWzQsMywiIiwyLHsic2hvcnRlbiI6eyJzb3VyY2UiOjIwLCJ0YXJnZXQiOjIwfX1dLFs2LDUsIiIsMix7InNob3J0ZW4iOnsic291cmNlIjoyMCwidGFyZ2V0IjoyMH19XV0=
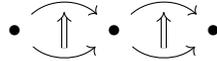
\[\begin{tikzcd}
    \bullet & \bullet & \bullet
    \arrow[""{name=0, anchor=center, inner sep=0}, curve={height=-12pt}, from=1-1, to=1-2]
    \arrow[""{name=1, anchor=center, inner sep=0}, curve={height=12pt}, from=1-1, to=1-2]
    \arrow[""{name=2, anchor=center, inner sep=0}, curve={height=-12pt}, from=1-2, to=1-3]
    \arrow[""{name=3, anchor=center, inner sep=0}, curve={height=12pt}, from=1-2, to=1-3]
    \arrow[shorten <=3pt, shorten >=3pt, Rightarrow, from=1, to=0]
    \arrow[shorten <=3pt, shorten >=3pt, Rightarrow, from=3, to=2]
  \end{tikzcd}
\]
We note that this is a pasting diagram which can be used to define the horizontal composition operation on  2-cells.
\end{example}

To simplify definitions of substitutions between wedge sums of pasting contexts, we will write substitutions diagrammatically by specifying the individual components. Indeed given substitutions \(\sigma : \Gamma \to \Gamma'\) and \(\tau : \Delta \to \Delta'\), such that $\sigma$ sends the last terminal $\star$\-typed variable of $\Gamma$ to that of $\Gamma'$, and $\tau$ sends the initial variable of $\Delta$ to that of $\Delta'$, one sees easily that there is a substitution \(\sigma \vee \tau : \Gamma \vee \Delta \to \Gamma' \vee \Delta'\) which we will depict as follows:
% https://q.uiver.app/?q=WzAsNixbMCwyLCJcXFNpZ21hXFxHYW1tYSJdLFsxLDIsIlxcdmVlIl0sWzIsMiwiXFxTaWdtYSBcXERlbHRhIl0sWzAsMCwiXFxTaWdtYSBcXEdhbW1hJyJdLFsyLDAsIlxcU2lnbWFcXERlbHRhJyJdLFsxLDAsIlxcdmVlIl0sWzAsMywiXFxTaWdtYSBcXHNpZ21hIl0sWzIsNCwiXFxTaWdtYSBcXHRhdSJdXQ==
\[\begin{tikzcd}[column sep=tiny, row sep=7pt]
    {\Gamma'} & \vee & {\Delta'} \\
    \\
    \Gamma & \vee & \Delta
    \arrow["{\sigma}", from=3-1, to=1-1, pos=.4]
    \arrow["{\tau}", from=3-3, to=1-3, pos=.4]
  \end{tikzcd}
\]

\mypara{Tree Contexts.}
We now have the machinery needed to define the context generated from a tree. Our definition of tree will be based on lists which we will write in square bracket notation $[x_1,\dots,x_n]$. We also use common list notations such as \(\_\doubleplus\_\) for concatenating two lists, \verb|[]| for the empty list, and \(n :: ns\) for the list with first element \(n\) and tail given by the list \(ns\).

We these conventions, we may define trees inductively as follows:
\begin{definition}
  A \emph{planar rooted tree} is  a (possibly empty) list of planar rooted trees.
\end{definition}

\noindent
Subtrees of a tree can be indexed by a list of natural numbers \(P\), giving a subtree \(T^P\), by letting \(T^{\verb|[]|} = T\) and \(T^{k::P} = {(T_k)}^P\) if \(T = [T_1, \dots, T_n]\).

Each tree generates a context, using the constructions of the previous subsections.

\begin{definition}
  For a tree \(T = [T_1,\dots,T_n]\), the context $\lfloor T \rfloor$ generated from it is given by:
  \[\lfloor T \rfloor := \bigvee\limits_{i = 1}^n \Sigma\lfloor T_i \rfloor\]
  Here we understand the convention that when \(T\) is the empty list, we get a singleton context of the form \(\_ : \star\). We will commonly abuse notation and omit the \(\lfloor - \rfloor\) operator and use trees as contexts when it will not cause confusion.
\end{definition}

By a simple induction using the properties of suspensions and wedge sums, we get that any context generated from a tree is a pasting context. The stronger result holds that \(\lfloor - \rfloor\) is an isomorphism between trees and pasting diagrams, though we omit the proof here as it will not be needed for the definition of insertion. Next we give some simple definitions on trees that will be needed later.

\begin{definition}
\label{def:treetrunk}
The dimension of a tree \(\dim(T)\) is \(0\) if \(T\) is empty or \(1 + \max_k{\dim(T_k)}\) if \(T = [T_1,\dots,T_n]\). For a tree \(T\), its \emph{trunk height}, \(\th(T)\), is \(1 + \th(T_1)\) if \(T = [T_1]\) and \(0\) otherwise. A tree is \emph{linear} if its trunk height equals its dimension.
\end{definition}

\begin{figure}
  \centering
  \begin{equation*}
      \begin{tikzpicture}[every node/.style={scale=0.6},baseline=(x11.base)]
                          \node [on grid](x01) {$\bullet$};
                          \node [above=0.5 of x01, on grid] (x11) {$\bullet$};
                          \node [above left=0.5 and 0.25 of x11, on grid] (x21) {$\bullet$};
                          \node [above right=0.5 and 0.25 of x11, on grid] (x22) {$\bullet$};
                          \draw (x01.center) to (x11.center);
                          \draw (x11.center) to (x21.center);
                          \draw (x11.center) to (x22.center);
                        \end{tikzpicture}
                        \qquad
                        \begin{tikzcd} [ampersand replacement=\&]
        {x} \&\& {y}
        \arrow[""{name=0, anchor=center, inner sep=0}, "{g}"{pos=0.7,description}, from=1-1, to=1-3]
        \arrow[""{name=1, anchor=center, inner sep=0}, "{f}"', curve={height=18pt}, from=1-1, to=1-3]
        \arrow[""{name=2, anchor=center, inner sep=0}, "{h}", curve={height=-18pt}, from=1-1, to=1-3]
        \arrow["{\beta}", shorten <=2pt, shorten >=2pt, Rightarrow, from=0, to=2]
        \arrow["{\alpha}", shorten <=2pt, shorten >=2pt, Rightarrow, from=1, to=0]
      \end{tikzcd}\qquad
        \vcenter{\hbox{\raisebox{-5pt}{\(x\)}{\((f\raisebox{5pt}{\(\alpha\)}g\raisebox{5pt}{\(\beta\)}h)\)}\raisebox{-5pt}{\(y\)}}}
                      \end{equation*}
                      \caption{  \label{fig:tree-example}
Example tree and identity labelling.}
\end{figure}

Under the bijection between trees and pasting contexts, a substitution $\sigma : \lfloor T \rfloor \to \Gamma$ from the context associated to $T$ to an arbitrary context $\Gamma$ may be represented by an appropriate type of \emph{labellings} of $T$ which we now define:

\begin{definition}[Tree Labelling]
  A labelling \(L : T \to \Gamma\) from a tree \(T = [T_1,\dots,T_n]\) to \(\Gamma\) is the following data:
  \[\vcenter{\hbox{\raisebox{-5pt}{\(t_0\)}{\(L_1\)}\raisebox{-5pt}{\(t_1\)}}}\ \cdots\ \vcenter{\hbox{\(L_n\)\raisebox{-5pt}{\(t_n\)}}}\]
  where each \(t_i\) is a term of \(\Gamma\) and each \(L_i : T_i \to \Gamma\) is itself a labelling of $T_i$ in $\Gamma$. The terms \(t_i\) label the 0-dimensional variables of the tree, with the labellings \(L_i\) assigning terms recursively to the higher dimensional variables of the tree. Every tree has an identity labelling \(\id_T : T \to T\).
\end{definition}

\begin{example}
  \cref{fig:tree-example} shows a graphical representation of the tree \verb|[[[][]]]| and the context it generates. It also gives the identity labelling on this tree.
\end{example}

\subsection{\Catt with Equality Rules}
\label{sec:equality}

\noindent
We conclude this section by extending $\Catt$ with a definitional equality determined by a set of \emph{equality generators} $\mathcal R$.  We denote the resulting theory by $\Catt_{\mathcal R}$. Formally, each generator $R \in \mathcal{R}$ is given by a triple $R =(\Gamma,s,t)$ where $\Gamma$ is a context and $s,t \in \Term_\Gamma$.  In what follows, we identify certain properties a set of generators $\mathcal{R}$ might enjoy which endow the resulting definitional equality with useful meta-theoretic behavior. We then give some preliminary examples. Our technical development in this section is driven by the approach taken in our formalisation.

To begin, we fix a set $\mathcal{R}$ of equality generators and add inductively defined equality judgments %which we refer to as ``definitional equality'':
\begin{alignat*}{2}
  &\Gamma \vdash s = t &\qquad&\text{Terms \(s, t : \Term_\Gamma\) are equal.}\\
  &\Gamma \vdash A = B &&\text{Types \(A, B : \Type_\Gamma\) are equal.}\\
  &\Gamma \vdash \sigma = \tau &&\text{Substitutions \(\sigma, \tau: \Delta \to \Gamma\) are equal.}
\end{alignat*}
These will be defined mutually inductively alongside the typing rules. Note that the definitional equality is not fibred over types, as it was not necessary for this development, and this setup allowed development of the formalisation to proceed~\cite{alex_rice_2024_11149192}. We also add the following new typing rule, named the \textit{conversion rule}:
\[ \inferrule{\Gamma \vdash s : A \\ \Gamma \vdash A = B}{\Gamma \vdash s : B}\]
\begin{figure}
  \centering
  \begin{mathpar}
    \inferrule{x \in \Gamma}{\Gamma \vdash x = x} \and
    \inferrule{\Gamma \vdash s = t}{\Gamma \vdash t = s} \and
    \inferrule{\Gamma \vdash s = t \\ \Gamma \vdash t = u}{\Gamma \vdash s = u} \and
    \inferrule{\Delta \vdash A = B \\ \Gamma \vdash \sigma = \tau}{\Gamma \vdash \Coh \Delta A \sigma = \Coh \Delta B \tau} \and
    \inferrule{ }{\Gamma \vdash \star = \star} \and
    \inferrule{\Gamma \vdash s = s' \\ \Gamma \vdash t = t' \\ \Gamma \vdash A = A'}{\Gamma \vdash \arr s A t = \arr {s'} {A'} {t'}}\and
    \inferrule{ }{\Gamma \vdash \langle  \rangle = \langle  \rangle}\and
    \inferrule{\Gamma \vdash \sigma = \tau \\ \Gamma \vdash s = t}{\Gamma \vdash \langle \sigma, x \mapsto s \rangle = \langle \tau, x \mapsto t \rangle}
  \end{mathpar}
  \caption{  \label{fig:def-eq}
Structural rules for definitional equality.}
\end{figure}%
% This allows us to type terms that would not be valid in the base theory \Catt. For example. this allows us to compose terms whose boundaries only agree up to definitional equality. instead of the stronger syntactic equality.
Finally, in addition to the structural rules given in \cref{fig:def-eq}, we add a family of rules:
\[ \inferrule{(\Gamma,s,t) \in \mathcal{R} \\ \Gamma \vdash s : A}{\Gamma \vdash s = t}\]
These rules are deliberately asymmetric; only the left hand side requires a proof of validity. Preempting \cref{sec:reduction}, this is because the equalities we use in our theories will take the form of a reduction, where the right hand side will be constructed from the left hand side of the equation.
We refer to the equality relation $=$ defined by these rules as \emph{definitional equality}.

We now identify some attractive properties that the equality rules \(\mathcal{R}\) can satisfy which make the resulting type theory well-behaved.
\mypara{Weakening Condition.}
  \(\mathcal{R}\) has the weakining condition if for all \((\Gamma,s,t) \in \mathcal{R}\) and \(A : \Type_\Gamma\) we have:
  \[((\Gamma, x : A), s, t) \in \mathcal{R}\]
  Note that even though \(s\) and \(t\) are terms over \(\Gamma\), they can be weakened to terms over \(\Gamma, A\), which we do not write explicitly here.

  This condition allows us to show that all equality and typing is preserved by context extension. It also gives us an easy proof that the identity substitution is well typed.

\mypara{Substitution Condition.}
  \(\mathcal{R}\) has the substitution condition if for all \((\Gamma,s,t) \in \mathcal{R}\) and \(\sigma : \Gamma \to \Delta\) with \(\Delta \vdash \sigma : \Gamma\) we have:
  \[(\Delta, s \sub \sigma, t \sub \sigma) \in \mathcal{R}\]
Note here that the typing precondition on \(\sigma\) is with respect to the typing generated from the rule set \(\mathcal{R}\). If \(\mathcal{R}\) satisfies the substitution condition then more generally we have that typing and equality is preserved by substitution. We also get that given \(\Delta \vdash \sigma = \tau\) and \(s : \Term_\Gamma\) that \(\Delta \vdash s \sub \sigma = s \sub \tau\), though this does not actually require the substitution condition.

  If the set of rules $\mathcal{R}$ satisfies the weakening and substitution conditions, then there is a well-defined quotient category of well-typed contexts and substitutions which we will denote (as a slight abuse of notation) by \(\mathsf{Catt}_{\mathcal{R}}\). Lastly we note that the substitution condition implies the weakening condition by using the weakening substitution \(\Gamma \to \Gamma , A\), but in the formalisation it was convenient for them to be given separately.

\mypara{Suspension Condition.}
  We say that \(\mathcal{R}\) has the suspension condition if \((\Gamma,s,t) \in \mathcal{R}\) implies:
  \[(\Sigma \Gamma, \Sigma s, \Sigma t) \in \mathcal{R}\]
  This is sufficient to show that typing and equality is respected by suspension.

  \begin{definition}
    A set of equality generators \(\mathcal{R}\) is \emph{tame} is it satisfies the weakening, substitution, and suspension conditions.
  \end{definition}
  In any tame theory, it can be shown that tree labellings and substitutions between wedge sums can be well typed, and it can also be shown that the standard constructions (type, term, and coherence) are valid.

\mypara{Support Condition.}
  \(\mathcal{R}\) has the support condition if all \((\Gamma,s,t) \in \mathcal{R}\) with \(\Gamma \vdash s : A\) we have \(\supp(s) = \supp(t)\). Unsurprisingly, this condition being true implies all equalities preserve support.

  While this rule may at first appear obvious to show, it turns out to be not quite so trivial. Despite knowing that \(s\) is valid, we have no guarantee that it is well behaved with respect to support, as it could contain equalities that do not preserve support. We therefore give the following lemma and proof strategy, which follows the method used in \cite{StrictUnits} to show preservation of support. We first define a set:
  \[\mathcal{R}_{\mathsf{s}} = \{(\Gamma, s, t) \in \mathcal{R}\ |\ \supp(s) = \supp(t)\}\]
  This generates a new type theory \(\Catt_{\mathcal{R}_\mathsf{s}}\). For clarity we let this type theory have judgments of the form \(\vdash_{\mathsf{s}}\).
  \begin{lemma}
    Suppose for all \((\Gamma,s,t) \in \mathcal{R}\) such that \(\Gamma \vdash_{\mathsf{s}} s : A\) for some \(A : \Type_\Gamma\) we have that \(\supp(s) = \supp(t)\). Then \(\mathcal{R}\) satisfies the support condition.
  \end{lemma}
  Using this lemma, we can make any constructions in \(\mathcal{R}\) in \(\Catt_{\mathcal{R}_{\mathsf{s}}}\) which will automatically give us that certain equalities preserve support. We will see this later with insertion, where we will show that the insertion operation is valid in any theory satisfying the appropriate conditions.

\mypara{Conversion Condition.}
  The last condition is the conversion condition that states that if \((\Gamma,s,t) \in \mathcal{R}\) then \(\Gamma \vdash s : A\) implies \(\Gamma \vdash t : A\). This along with the support condition is enough to show that equality preserves typing.

We can immediately see that \(\Catt = \Catt_\emptyset\), and since \(\emptyset\) trivially satisfies the above conditions, all the results above hold for \(\Catt\) itself.

% Furthermore, one can easily see that if $\mathcal{R}_1$ and $\mathcal{R}_2$ are two sets of rules which satisfy any of the above conditions, the so does $\mathcal{R}_1 \cup \mathcal{R}_2$: this is simply because adding new rules always generates \emph{new} definitional equalities, resulting in a coarser relation and preserving the required conditions. \todo{I forgot that adding more rules also adds more well typed terms} As a consequence, for any finite set of rules $\mathcal{R}$ it suffices to verify any of the above conditions for each rule in isolation.

\mypara{Disc Removal.}
We give two examples of equality rules from \Cattsu \cite{StrictUnits} which will be reused for \Cattsua. The first trivialises unary compositions.
\begin{definition}[Disc removal]
  Recall that any substitution from a disc is of the form \(\{A,t\}\) for some term \(t\) and type \(A\). Disc removal equates the terms \(\mathcal{C}_{D^n}^n\sub{\{A,s\}}\) and \(s\) in context \(\Gamma\) for any \(n\), \(s : \Term_\Gamma\), and \(A : \Type_\Gamma\). In other words, disc removal gives us equalities of the following form, after unwrapping the constructions above:
  \[\Gamma \vdash \Coh {D^n} {U_{D^n}^n} {\{A, s\}} = s\]
\end{definition}

\noindent
This can be intuitively understood as saying ``the term $s$ is equal to $(s)$, the unary composite of $s$''.

Disc removal gives the property that \(\mathcal{C}_\Delta^n = \mathcal{T}_\Delta^n\) for \(n > 0\), effectively removing the need to  differentiate between the two when working up to definitional equality.

\mypara{Endo-Coherence Removal.}
The second equality rule simplifies a class of terms called ``endo-coherences''. These are terms of the following form:
\[ \Coh \Delta {\arr s A s} \sigma\]
These are ``coherence laws'' that can be understood intuitively as saying ``the term $s[\sigma]$ is equal to the term $s[\sigma]$''. But we already have a canonical way to express that --- the identity on $s[\sigma]$. This inspires the following reduction.

\begin{definition}[Endo-coherence removal]
  Endo-coherence removal is the following class of equalities:
  \[ \Gamma \vdash \Coh \Delta {\arr s A s} \sigma = \mathbbm{1}\sub{\{A \sub \sigma, s \sub \sigma\}}\]
  for all \(\Delta\), \(s\), \(A\), and \(\sigma\).
\end{definition}

It can be checked (using proofs from \cite{StrictUnits}) that both disc removal and endo-coherence removal satisfy all the conditions listed above. By defining the rule-set \[\su = \text{disc removal} \cup \text{endo-coherence removal} \cup \text{pruning}\] the type theory \(\Cattsu\) is recovered. This establishes that \Cattsu is part of the general schema presented in this section.

\subsection{Models of \texorpdfstring{\(\Catt_{\mathcal{R}}\)}{CattR}}
\label{sec:models-cattr}

The models of \(\Catt_{\mathcal{R}}\) are defined in the same way as the models of \Cattsu, and many of the results proved in \cite{StrictUnits} apply directly to well-behaved variants of \(\Catt_{\mathcal{R}}\). For completeness, we summarise these results here. Recalling the definition of the syntactic category \(\mathsf{Catt}_{\mathcal{R}}\), we can define the models of \(\Catt_{\mathcal{R}}\).
\begin{definition}
  Let \(\mathcal{R}\) be a tame set of equality generators. A \emph{model} of \(\Catt_{\mathcal{R}}\) is a presheaf on \(\mathsf{Catt}_{\mathcal{R}}\) for which the opposite functor \(\mathsf{Catt}_{\mathcal{R}} \to \mathbf{Set}^{\text{op}}\) preserves globular sums.
\end{definition}

\noindent
We define a \(\mathcal{R}\)-semistrict \(\infty\)-category to be one of these models. When \(\mathcal{R}\) satisfies the support condition,  the \emph{rehydration} algorithm from \cite{StrictUnits} allows a \Catt term \(t\) to be obtained from a \(\Catt_{\mathcal{R}}\) term \(s\) such that \(s =_{\mathcal{R}} t\). This implies that the obvious functor \(\mathsf{Catt} \to \mathsf{Catt}_{\mathcal{R}}\) is full, and hence being \(\mathcal{R}\)-semistrict is a \textit{property} of \(\infty\)\-categories, as one would expect.

Natural transformations between these presheaves correspond to strict functors. Weak functors between \(\infty\)-categories can be constructed from these strict functors by a well-known procedure due to Garner~\cite{GARNER20102269}.

\section{Insertion}
\label{sec:insertion}

\noindent
The semistrict behaviour of our type theory \Cattsua, our adaptation  of \Catt with strict units and associators, is principally driven by a new equality rule called ``insertion''. This equality rule incorporates part of the structure of an argument context into the head context, simplifying the overall syntax of the term.

To be a candidate for insertion, an argument must not occur as the source or target of another argument of the term, and we call such arguments \emph{locally maximal}. Consider the composite \mbox{\(f \cdot (g \cdot h)\)}. This term has two locally maximal arguments, \(f\) and \(g \cdot h\), the second of which is an (standard) coherence. Insertion allows us to merge these two composites into one by ``inserting'' the pasting diagram of the inner coherence into the pasting diagram of the outer coherence. In the case above we will get that the term \(f \cdot (g \cdot h)\) is equal to the ternary composite \(f \cdot g \cdot h\), a term with a single coherence. As the term \((f \cdot g) \cdot h\) also reduces by insertion to the ternary composite, we see that both sides of the associator become equal under insertion. The action of insertion on these contexts is shown in \cref{fig:insertion}.

\begin{figure}
$$
\begin{aligned}
\begin{tikzpicture}
\node (x) at (0,0)  {$x$};
\node (y) at (1,0) {$y$};
\node (z) at (2,0) {$z$};
\draw [->] (x) to node [above, font=\small] {$f$} (y);
\draw [->] (y) to node [above, font=\small] {$g'$} (z);
\begin{scope}[xshift=.5cm, yshift=1.5cm, red]
%\draw [fill=red!10, draw=none] (1,0.05) ellipse (1.2cm and .6cm);
\draw [rounded corners, fill=red!7, draw=none] (-.25,-.35) rectangle +(2.5,1);
\node (x2) at (0,0)  {$x'$};
\node (y2) at (1,0) {$y'$};
\node (z2) at (2,0) {$z'$};
\draw [->] (x2) to node [above, font=\small] {$g$} (y2);
\draw [->] (y2) to node [above, font=\small] {$h$} (z2);
\end{scope}
\draw [->, thick, red] (1.5,1) to +(0,-.5);
\end{tikzpicture}
\end{aligned}
\quad\leadsto\quad
\begin{aligned}
\begin{tikzpicture}
\node (x) at (0,0)  {$x \vphantom'$};
\node [red] (y) at (1,0) {$x'$};
\node [red] (z) at (2,0) {$y'$};
\node [red] (w) at (3,0) {$z'$};
\begin{scope}[xshift=.5cm, yshift=1.5cm, red]
\draw [rounded corners, fill=white, draw=none] (-.25,-.35) rectangle +(2.5,1);
\end{scope}
\draw [->] (x) to node [above, font=\small] {$f$} (y);
\draw [->, red] (y) to node [above, font=\small] {$g$} (z);
\draw [->, red] (z) to node [above, font=\small] {$h$} (w);
\end{tikzpicture}
\end{aligned}
$$
\caption{\label{fig:insertion}
Insertion acting on the composite \(f \cdot (g \cdot h)\)}
\end{figure}

More generally we consider a coherence term \(\Coh \Delta A \sigma : \Term_\Gamma\), where there is some locally maximal variable \(x : A \in \Delta\) such that \(x \sub \sigma\) is itself an standard coherence \(\mathcal{C}_\Theta^n \sub \tau\). Under certain conditions on the shape of \(\Theta\) and \(\Delta\) (which will be specified in \cref{sec:construction}) we
will construct the following data as part of the insertion operation:
\begin{itemize}
\item The \textit{inserted context} \(\insertion \Delta x \Theta\), obtained by inserting \(\Theta\) into \(\Delta\) along \(x\). The inserted context is a pasting diagram.
\item The \textit{interior substitution} \(\iota : \Theta \to \insertion \Delta x \Theta\), the inclusion of \(\Theta\) into a copy of \(\Theta\) living in the inserted context.
\item The \textit{exterior substitution} \(\kappa : \Delta \to \insertion \Delta x \Theta\), which maps \(x\) to standard coherence over the copy of \(\Theta\), or more specifically \(\mathcal{C}_\Theta^n \sub \iota\), and other locally maximal variables to their copy in the inserted context.
\item The \textit{inserted substitution} \(\insertion \sigma x \tau : \insertion \Delta x \Theta \to \Gamma\), which collects the appropriate parts of \(\sigma\) and \(\tau\).
\end{itemize}
Using this notation, insertion yields the following reduction:
\[\Coh \Delta A \sigma \leadsto \Coh {\insertion \Delta x \Theta} {A \sub \kappa} {\insertion \sigma x \tau}\]
These constructions can be assembled into the following diagram:
% https://q.uiver.app/?q=WzAsNSxbMCwwLCJEX24iXSxbMSwwLCJcXERlbHRhIl0sWzAsMSwiXFxUaGV0YSJdLFsxLDEsIlxcaW5zZXJ0aW9uIFxcRGVsdGEgeCBcXFRoZXRhIl0sWzIsMiwiXFxHYW1tYSJdLFsxLDMsIlxca2FwcGEiXSxbMiwzLCJcXGlvdGEiLDJdLFswLDEsIlxce0EseFxcfSJdLFswLDIsIlxce1xcbWF0aGNhbHtVfV9cXFRoZXRhXm4sIFxcbWF0aGNhbHtDfV9cXFRoZXRhXm5cXH0iLDJdLFsxLDQsIlxcc2lnbWEiLDAseyJjdXJ2ZSI6LTN9XSxbMiw0LCJcXHRhdSIsMix7ImN1cnZlIjoyfV0sWzMsNCwiXFxpbnNlcnRpb24gXFxzaWdtYSB4IFxcdGF1IiwxXSxbMywwLCIiLDEseyJzdHlsZSI6eyJuYW1lIjoiY29ybmVyIn19XV0=
\[\begin{tikzcd}
    {D^n} & \Delta \\
    \Theta & {\insertion \Delta x \Theta} \\
    && \Gamma
    \arrow["\kappa", from=1-2, to=2-2]
    \arrow["\iota"', from=2-1, to=2-2]
    \arrow["{\{A,x\}}", from=1-1, to=1-2]
    \arrow["{\{\mathcal{U}_\Theta^n, \mathcal{C}_\Theta^n\}}"', from=1-1, to=2-1]
    \arrow["\sigma", curve={height=-18pt}, from=1-2, to=3-3]
    \arrow["\tau"', curve={height=12pt}, from=2-1, to=3-3]
    \arrow["{\insertion \sigma x \tau}"{description}, from=2-2, to=3-3]
    \arrow["\lrcorner"{anchor=center, pos=0.125, rotate=180}, draw=none, from=2-2, to=1-1]
  \end{tikzcd}
\]
The commutativity of the outer boundary is the equation \(x \sub \sigma = \mathcal{C}_\Theta^n \sub \tau\), one of the conditions for the construction. The commutativity of the inner square is a property of the exterior substitution as stated above. Furthermore, as suggested by the diagram, \(\insertion \Delta x \Theta\) is a pushout, and \(\insertion \sigma x \tau\) is unique map to \(\Gamma\) determined by the universal property of the pushout. This gives a different way to think of the insertion, and gives the intuition that insertion is the result of taking the disjoint union of the two contexts, and gluing together \(x\) in the first with the standard coherence over the second.

\subsection{The Insertion Construction}
\label{sec:construction}

\noindent
We have stated the existence of inserted contexts and associated maps, and claimed that they satisfy a universal property. In this section we give a direct construction of these objects. All constructions will be done by induction over the tree structure of pasting diagrams, which were introduced in \cref{sec:trees}. Trees are more convenient for our technical development than contexts, and so we will work with trees throughout.

To allow us to proceed with inductive definitions we need an inductive version of a locally maximal variable, which we will call a \emph{branch}. We define some properties of branches as follows, illustrated in \cref{fig:leafheight}.

\begin{figure}
$$
\begin{tikzpicture}[scale=.8, yscale=.9]
\draw (0,0) node [dot] {} to (-1,1) node [dot] {} to (0,2) node [dot] {};
\draw (-1,1) to (-2,2) node [dot] {} to (-2,3) node [dot] {};
\draw (0,0) to (1,1) node [dot] {};
\draw (0,0) to (0,-1) node [dot] {};
\node [left] at (-2,2) {$T^P$};
\node [left] at (-2,3) {$\lfloor P \rfloor$};
\draw [|->] (-3,-1) to node [right] {$\bh(P)$} +(0,2);
\draw [|->] (-1.5,-1) to node [right] {$\th(T)$} +(0,1);
\draw [|->] (-4.5,-1) to node [right] {$\lh(P)$} +(0,4);
\end{tikzpicture}
$$
\caption{\label{fig:leafheight} Leaf height, branch height and trunk height.}
\end{figure}

\begin{definition}
  A \emph{branch} \(P\) of a tree \(T\) is a non-empty indexing list for a subtree of \(T\) which is linear. A branch \(P\) of \(T\) gives a locally maximal variable \(\lfloor P \rfloor\) of \(\lfloor T \rfloor\) by taking the unique locally maximal variable of \(\lfloor T_P \rfloor\). Define the \emph{branch height} of \(P\), denoted \(\bh(P)\), to be one less that the length of \(P\) (note that \(P\) is always non-empty). Finally define the \emph{leaf height} \(\lh(P)\) of a path \(P\) as the dimension of~\(\lfloor P \rfloor\). As with trees, we will omit the \(\lfloor - \rfloor\) notation and use a branch as a variable when it is clear.
\end{definition}

For every locally maximal variable, there is some branch representing it, though not necessarily a unique one. Recall \cref{def:treetrunk} of trunk height for a tree. We now give one of the central definitions of the paper, which was also given an informal exposition in the introduction, and illustrated with \cref{fig:insertionintro}.

\begin{definition}[Inserted Tree]
  Given trees \(S\) and \(T\), and a branch \(P\) of \(S\) such that \(\th(T) \geq \bh(P)\), we define the \emph{inserted tree} \(\insertion S P T\) by induction on the length of \(P\):
  \begin{itemize}
  \item Suppose \(P = [k]\) and \(S = [S_1,\dots,S_k,\dots,S_n]\). Then:
    \[\insertion S P T = [S_1,\dots,S_{k-1}] \doubleplus T \doubleplus [S_{k+1},\dots,S_n]\]
  \item Suppose \(P\) has length greater than 1 so that \(P = k :: P'\) and again \(S = [S_1,\dots,S_k,\dots,S_n]\). We note that \(P'\) is a branch of \(S_k\) and by the condition on trunk height of \(T\) we have \(T = [T_1]\). Then:
    \[\insertion S P T = [S_1,\dots,S_{k-1},\insertion {S_k} {P'} {T_1},S_{k+1},\dots,S_n ] \]
  \end{itemize}
  We draw attention to the condition of the trunk height of \(T\) being at least the branch height of \(P\), which is necessary for the induction to proceed. We recall that a tree is identified with a list of trees, and that in the first case of insertion \(T\) is treated as a list, and in the second case \(\insertion {S_k} {P'} {T_1}\) is treated as a single tree which forms one of the subtrees of \(\insertion S P T\).
\end{definition}

We now proceed to define the interior and exterior substitutions, which will be done using the diagrammatic notation introduced in \cref{sec:trees}.

\begin{definition}[Interior Substitution]
  Given \(S, T\) trees, with \(P\) a branch of \(S\) with \(\th(T) \geq \bh(P)\) we define the interior substitution \(\iota_{S,P,T} : T \to \insertion S P T\) by induction on~\(P\).

  \(\bullet\) When \(P = [k]\), \(S = [S_1,\dots,S_k,\dots,S_n]\) we get:
    % https://q.uiver.app/?q=WzAsNixbMCwwLCJTXzBcXHZlZVxcZG90c1xcdmVlIFNfe2stMX0iXSxbNCwwLCJTX3trKzF9IFxcdmVlIFxcZG90cyBcXHZlZSBTX24iXSxbMiwwLCJUIl0sWzMsMCwiXFx2ZWUiXSxbMSwwLCJcXHZlZSJdLFsyLDIsIlQiXSxbNSwyLCJcXGlkIl1d
    \[\begin{tikzcd}[column sep=smaller,cramped,row sep=6pt]
        {[S_1,\dots,S_{k-1}]} & \vee & T & \vee & {[S_{k+1},\dots,S_n]} \\
        \\
        && T
        \arrow["\id"{font = \normalsize}, from=3-3, to=1-3]
      \end{tikzcd}
    \]

    \(\bullet\) When \(P = k :: P'\), \(S = [S_1,\dots,S_k,\dots,S_n]\) we get:
    % https://q.uiver.app/?q=WzAsNixbMCwwLCJcXGxmbG9vciBbU18xLFxcZG90cyxTX3trLTF9XSBcXHJmbG9vciJdLFs0LDAsIlxcbGZsb29yIFtTX3trKzF9LFxcZG90cyxTX25dIFxccmZsb29yIl0sWzIsMCwiXFxTaWdtYSBcXGxmbG9vciBcXGluc2VydGlvbiB7U19rfSB7UCd9IHtUXzF9IFxccmZsb29yIl0sWzMsMCwiXFx2ZWUiXSxbMSwwLCJcXHZlZSJdLFsyLDIsIlxcU2lnbWEgXFxsZmxvb3IgVF8xIFxccmZsb29yIl0sWzUsMiwiXFxTaWdtYSBcXGlvdGFfe1NfayxQJyxUXzF9Il1d
    \[\begin{tikzcd}[column sep=smaller, cramped, row sep=6pt]
        {[S_1,\dots,S_{k-1}]} & \vee & {\Sigma \insertion {S_k} {P'} {T_1}} & \vee & {[S_{k+1},\dots,S_n]} \\
        \\
        && {\Sigma T_1}
        \arrow["{\Sigma \iota_{S_k,P',T_1}}"{font = \normalsize}, from=3-3, to=1-3]
      \end{tikzcd}
    \]
We may drop the subscripts on \(\iota\) when they are easily inferred.
\end{definition}

\begin{definition}[Exterior Substitution]
  Given \(S, T\) trees, with \(P\) a branch of \(S\) with \(\th(T) \geq \bh(P)\) we define the exterior substitution \(\kappa_{S,P,T} : S \to \insertion S P T\) by induction on \(P\).

  \(\bullet\) When \(P = [k]\), \(S = [S_1,\dots,S_k,\dots,S_n]\) we get:
    % https://q.uiver.app/?q=WzAsMTAsWzAsMCwiXFxsZmxvb3IgW1NfMSxcXGRvdHMsU197ay0xfV0gXFxyZmxvb3IiXSxbNCwwLCJcXGxmbG9vciBbU197aysxfSxcXGRvdHMsU19uXSBcXHJmbG9vciJdLFsyLDAsIlxcbGZsb29yIFQgXFxyZmxvb3IiXSxbMywwLCJcXHZlZSJdLFsxLDAsIlxcdmVlIl0sWzIsMiwiXFxTaWdtYSBcXGxmbG9vciBTX2sgXFxyZmxvb3IiXSxbMCwyLCJcXGxmbG9vciBbU18xLFxcZG90cyxTX3trLTF9XFxyZmxvb3IiXSxbMSwyLCJcXHZlZSJdLFszLDIsIlxcdmVlIl0sWzQsMiwiXFxsZmxvb3IgW1Nfe2srMX0sXFxkb3RzLFNfbl0gXFxyZmxvb3IiXSxbNSwyLCJcXHtcXG1hdGhjYWx7VX1fVF5uLCBcXG1hdGhjYWx7Q31fVF5uXFx9Il0sWzYsMCwiXFxpZCJdLFs5LDEsIlxcaWQiXV0=
    \[\begin{tikzcd}[column sep=smaller,cramped, row sep = 6pt]
        {[S_1,\dots,S_{k-1}]} & \vee & {T} & \vee & {[S_{k+1},\dots,S_n]} \\
        \\
        {[S_1,\dots,S_{k-1}]} & \vee & {\Sigma S_k} & \vee & {[S_{k+1},\dots,S_n]}
        \arrow["{\{\mathcal{U}_T^n, \mathcal{C}_T^n\}}"{font = \normalsize, pos=.4}, from=3-3, to=1-3]
        \arrow["\id"{font = \normalsize}, from=3-1, to=1-1]
        \arrow["\id"{font = \normalsize}, from=3-5, to=1-5]
      \end{tikzcd}\]
    Where we note that by the condition of \(P\) being a branch we have that \(S_k\) is linear and so \(\Sigma \lfloor S_k \rfloor\) is a disc.

    \(\bullet\) When \(P = k :: P'\), \(S = [S_1,\dots,S_k,\dots,S_n]\) we get:
    % https://q.uiver.app/?q=WzAsMTAsWzAsMCwiXFxsZmxvb3IgW1NfMSxcXGRvdHMsU197ay0xfV0gXFxyZmxvb3IiXSxbNCwwLCJcXGxmbG9vciBbU197aysxfSxcXGRvdHMsU19uXSBcXHJmbG9vciJdLFsyLDAsIlxcU2lnbWEgXFxsZmxvb3IgXFxpbnNlcnRpb24ge1Nfa30ge1AnfSB7VF8xfSBcXHJmbG9vciJdLFszLDAsIlxcdmVlIl0sWzEsMCwiXFx2ZWUiXSxbMiwyLCJcXFNpZ21hIFxcbGZsb29yIFNfayBcXHJmbG9vciJdLFswLDIsIlxcbGZsb29yIFtTXzEsXFxkb3RzLFNfe2stMX1cXHJmbG9vciJdLFsxLDIsIlxcdmVlIl0sWzMsMiwiXFx2ZWUiXSxbNCwyLCJcXGxmbG9vciBbU197aysxfSxcXGRvdHMsU19uXSBcXHJmbG9vciJdLFs1LDIsIlxcU2lnbWEgXFxrYXBwYV97U19rLFAnLFRfMX0iXSxbNiwwLCJcXGlkIl0sWzksMSwiXFxpZCJdXQ==
    \[\begin{tikzcd}[column sep=smaller, cramped, row sep = 6pt]
        {[S_1,\dots,S_{k-1}]} & \vee & {\Sigma \insertion {S_k} {P'} {T_1}} & \vee & {[S_{k+1},\dots,S_n]} \\
        \\
        {[S_1,\dots,S_{k-1}]} & \vee & {\Sigma S_k} & \vee & {[S_{k+1},\dots,S_n]}
        \arrow["{\Sigma \kappa_{S_k,P',T_1}}"{font=\normalsize}, from=3-3, to=1-3]
        \arrow["\id"{font=\normalsize}, from=3-1, to=1-1]
        \arrow["\id"{font=\normalsize}, from=3-5, to=1-5]
      \end{tikzcd}\]
Again the subscripts on \(\kappa\) may be dropped where they can be inferred.
\end{definition}

Lastly we define the inserted substitution.

\begin{definition}[Inserted Substitution]
  Given \(S, T\) trees, with \(P\) a branch of \(S\) with \(\th(T) \geq \bh(P)\) and \(\sigma : {S} \to \Gamma\), \(\tau : T \to \Gamma\), we define the \emph{inserted substitution} \(\insertion \sigma P \tau : {\insertion S P T} \to \Gamma\). Without loss of generality, we can assume that \(\sigma\) and \(\tau\) are given by labellings \(L,M\) of \(S\) and \(T\), and that we need to provide a labelling \(\insertion L P M : \insertion S P T \to \Gamma\). Let
  \[ S = [S_1,\dots,S_n] \qquad L = \vcenter{\hbox{\raisebox{-5pt}{\(s_0\)}{\(L_1\)}\raisebox{-5pt}{\(s_1\)}}}\ \cdots\ \vcenter{\hbox{\(L_n\)\raisebox{-5pt}{\(s_n\)}}}\]
  and then proceed by induction on \(P\).

  \(\bullet\) Let \(P = [k]\), and
  \[ T = [T_1,\dots,T_m] \qquad M = \vcenter{\hbox{\raisebox{-5pt}{\(t_0\)}{\(M_1\)}\raisebox{-5pt}{\(t_1\)}}}\ \cdots\ \vcenter{\hbox{\(M_m\)\raisebox{-5pt}{\(t_m\)}}}\]
    Then define \(\insertion L {[k]} M\) to be:
    \[\vcenter{\hbox{\raisebox{-5pt}{\(s_0\)}{\(L_1\)}\raisebox{-5pt}{\(s_1\)}}}\ \cdots\ \vcenter{\hbox{\(L_{k-1}\)\raisebox{-5pt}{\(t_0\)}\(M_1\)\raisebox{-5pt}{\(t_1\)}}}\ \cdots\  \vcenter{\hbox{\(M_m\)\raisebox{-5pt}{\(t_m\)}\(L_{k+1}\)\raisebox{-5pt}{\(s_{k+1}\)}}}\ \cdots\ \vcenter{\hbox{\(L_n\)\raisebox{-5pt}{\(s_n\)}}}\]

    \(\bullet\) Suppose \(P = k :: P'\) so that
    \[T = [T_1] \qquad M = \vcenter{\hbox{\raisebox{-5pt}{\(t_0\)}\(M_1\)\raisebox{-5pt}{\(t_1\)}}}\]
    Define \(\insertion L P M\) as:
    \[\vcenter{\hbox{\raisebox{-5pt}{\(s_0\)}{\(L_1\)}\raisebox{-5pt}{\(s_1\)}}}\ \cdots\ \vcenter{\hbox{{\(L_{k-1}\)}\raisebox{-5pt}{\(t_0\)}\((\insertion {L_k} {P'} {M_1})\)\raisebox{-5pt}{\(t_1\)}\(L_{k+1}\)\raisebox{-5pt}{\(s_{k+1}\)}}}\ \cdots\ \vcenter{\hbox{\(L_n\)\raisebox{-5pt}{\(s_n\)}}}\]
The inserted substitution is then defined as the substitution corresponding to this labelling.
\end{definition}

As we need a lot of data to perform an insertion, we will package it up to avoid repetition.

\begin{definition}
  An \emph{insertion point} is a triple \((S,P,T)\) where \(S\) and \(T\) are trees and \(P\) is a branch of \(S\) with \(\bh(P) \leq \th(T)\) and \(\lh(S) \geq \dim(T)\). An \emph{insertion redex} is a sextuple \((S,P,T,\Gamma,\sigma,\tau)\) where \((S,P,T)\) is an insertion point, \(\sigma : S \to \Gamma\) and \(\tau : T \to \Gamma\) are substitutions, and \(\lfloor P \rfloor\sub \sigma \equiv \mathcal{C}_T^{\lh(P)}\sub\tau\).
\end{definition}

\subsection{Properties of Insertion}
\label{sec:insert-props}

\noindent
Here all typing judgments are taken with respect to a tame set of rules $\mathcal R$ which are taken as implicit. We will state various properties of the constructions in the previous section in any such \(\Catt_{\mathcal{R}}\).

\begin{prop}
  \label{prop:insertion-typing}
  If \((S,P,T)\) is an insertion point then:
  \[ \insertion S P T \vdash \iota : T \quad \insertion S P T \vdash \kappa : S\]
  If further \((S,P,T,\Gamma,\sigma,\tau)\) is an insertion redex then:
  \[ \Gamma \vdash \insertion \sigma P \tau : \insertion S P T\]
\end{prop}
\begin{proof}
  % All substitutions are built from standard constructions, and their typing follows from the typing of these constructions, working by induction on \(P\). The equality \(\lfloor P \rfloor \sub \sigma \equiv \mathcal{C}_T^{\lh(P)} \sub \tau\) is needed to show that the generated labelling is valid.
  See \small\texttt{Catt/Tree/Insertion/Typing.agda(28)}.
\end{proof}

\begin{lemma}
  \label{lem:insertion-sub-susp}
  For any insertion redex \((S,P,T,\Gamma,\sigma,\tau)\):
  \[\Sigma (\insertion \sigma P \tau) \equiv \insertion {\Sigma \sigma} {0 :: P} {\Sigma \tau} \]
  Given another substitution \(\mu : \Gamma \to \Delta\), we have:
  \[(\insertion \sigma P \tau) \circ \mu \equiv \insertion {(\sigma \circ \mu)} P {(\tau \circ \mu)} \]
  and the constructions above are well defined.
\end{lemma}
\begin{proof}
  See \small\texttt{Catt/Tree/Insertion/Properties.agda(218)}.
\end{proof}

The following lemma about insertion into a disc is crucial for showing the reduction we define in \cref{sec:reduction} agrees with our equality.

\begin{lemma}
  \label{lem:disc-insertion-1}
  Let \(T\) be a tree, \(n \geq \dim (T)\), and \(P\) a branch of \(D^n\) with \(\bh(P) \leq \th(T)\). Then \(\insertion {D^n} P T = T\) and \(\iota_{D^n,P,T} \equiv \id\). Suppose further that \(\sigma : D^n \to \Gamma\) and \(\tau : T \to \Gamma\). Then \(\insertion \sigma P \tau \equiv \tau\).
\end{lemma}
\begin{proof}
  See \small\texttt{Catt/Tree/Insertion/Properties.agda(151)}.
\end{proof}

We next state the required conditions for the universal property of insertion. All of these can be proven by induction on \(P\).
\begin{lemma}
  \label{lem:insertion-subs-comm}
  For all insertion points \((S,P,T)\), the terms \(P \sub \kappa\) and \(\mathcal{C}_T^{\lh(P)} \sub \iota\) are syntactically equal.
  If we extend to an insertion redex \((S,P,T,\Gamma,\sigma,\tau)\) then the following hold:
  \begin{align*}
   \iota_{S,P,T} \circ (\insertion \sigma P \tau) &= \tau
   &
   \kappa_{S,P,T} \circ (\insertion \sigma P \tau) &= \sigma
   \end{align*}
 \end{lemma}
 \begin{proof}
   See \small\texttt{Catt/Tree/Insertion/Properties.agda(108,233)}.
 \end{proof}

\noindent
This leads us to the following theorem, proving that insertion arises as a pushout.

\begin{theorem}
  The following diagram is a pushout in \(\mathsf{Catt}_{\mathcal{R}}\) for any insertion point \((S,P,T)\), where \(A\) is the type of \(\lfloor P \rfloor\):
  \[\begin{tikzcd}
      {D^n} & S \\
      T & {\insertion S P T}
      \arrow["\kappa", from=1-2, to=2-2]
      \arrow["\iota"', from=2-1, to=2-2]
      \arrow["{\{A,P\}}", from=1-1, to=1-2]
      \arrow["{\{\mathcal{U}_\Theta^{\lh(P)}, \mathcal{C}_\Theta^{\lh(P)}\}}"', from=1-1, to=2-1]
      \arrow["\lrcorner"{anchor=center, pos=0.125, rotate=180}, draw=none, from=2-2, to=1-1]
    \end{tikzcd}
  \]

\end{theorem}
\begin{proof}
  All that is left to show after \cref{lem:insertion-subs-comm} is that the inserted substitution is the unique substitution satisfying the commutativity conditions. This is done by realising that each variable of \(\insertion S P T\) is either the image of a variable in \(S\) or~\(T\).
\end{proof}

\subsection{The Type Theory \texorpdfstring\Cattsua{Cattsua}}

\noindent
All the ingredients are now in place to define our principal type theory \Cattsua. We first formally define the insertion rule.

\begin{definition}[Insertion]
  The \emph{insertion rule} says that the following equation holds:
  \[\Gamma \vdash \Coh S A \sigma = \Coh {\insertion S P T} {A \sub \kappa} {\insertion \sigma P \tau}\]
  for all insertion redexes \((S,P,T,\Gamma,\sigma,\tau)\) and types \(A : \Type_S\).
\end{definition}

We now define the set of rules \(\sua\) to be the union of insertion, disc removal, and endo-coherence removal, and let \Cattsua be the type theory generated from these rules. As disc removal and endo-coherence removal satisfy all conditions, it remains to check that insertion also does. The weakening condition is trivial, as insertion does not interact with the ambient context the terms exist in. Substitution and suspension conditions follow quickly from \cref{lem:insertion-sub-susp} and some computation. This leaves support and conversion.

For support we appeal to the proof strategy detailed in \cref{sec:equality}. Suppose
\[\Gamma \vdash \Coh S A \sigma = \Coh {\insertion S P T} {A \sub \kappa} {\insertion \sigma P \tau}\]
is a valid insertion under the insertion rule, and further that \(\Gamma \vdash_{\mathsf{s}} \Coh S A \sigma : B\) (the left hand side of the rule is well typed in \(\Catt_{\mathcal{R}_{\mathsf{s}}}\)). Then, \cref{lem:insertion-subs-comm} holds in \(\Catt_{\mathcal{R}_{\mathsf{s}}}\) which implies \(\supp(\kappa \circ (\insertion \sigma P \tau)) = \supp(\sigma)\), and since \(\kappa\) is full (its support is the entire context), it follows that \(\supp(\insertion \sigma P \tau) = \supp(\sigma)\), as required.

Lastly, conversion follows from \cref{prop:insertion-typing}. It is also necessary to show that the support conditions hold in the generated term, which is done in the formalisation\footnote{\texttt{Catt/Typing/Insertion/Support.agda(214)}}.

The following property holds in \Cattsua:

\begin{theorem}
  \label{lem:standard-type-exterior}
  Let \((S,P,T)\) be an insertion point. Then for all \(n \leq \dim(S) + 1\), \(\mathcal{U}_S^n\sub{\kappa_{S,P,T}} = \mathcal{U}_{\insertion S P T}^n\) and if \(\dim(S) = n\) then \(\mathcal{T}_S^n \sub{\kappa_{S,P,T}} = \mathcal{T}_{\insertion S P T}\).
\end{theorem}
\begin{proof}
  See \small\texttt{Catt/Typing/Insertion/Equality.agda(319)}.
\end{proof}

\begin{cor}
  \label{cor:standard-insert}
  An insertion into an standard coherence is equal to an standard coherence over the inserted context. More specifically:
  \[\Gamma \vdash \Coh {\insertion S P T} {\mathcal{U}_S^n\sub{\kappa}} {\insertion \sigma P \tau} = \mathcal{C}_{\insertion S P T}^n \sub {\insertion \sigma P \tau}\]
   for any insertion redex \((S,P,T,\Gamma,\sigma,\tau)\) and \(n \geq \dim(T)\).
 \end{cor}

%\begin{remark}We end the section with the observation that inserting an identity is a very similar operation to pruning from \(\mathsf{Catt_{su}}\). Although the operations are not identical, we conjecture that any two terms identified in \Cattsu are also identified in \Cattsua. A full proof of this result would require a translation between the Dyck words used for pruning and trees used for insertion.\end{remark}

\section{A Decision Procedure for \texorpdfstring\Cattsua{Cattsua}}

\noindent
We show that equality for \(\Cattsua\) is decidable and hence type checking is also decidable. This gives an algorithm for checking validity of \(\Cattsua\) terms, which we have implemented and discuss in \cref{sec:implementation}.

Decidability is shown providing a reduction relation \(\leadsto\), such that the symmetric transitive reflexive closure of \(\leadsto\) agrees with equality on the set of valid terms. It is also shown that \(\leadsto\) is terminating, meaning for each term \(t\) we can generate a normal form \(N(t)\), and confluent which implies uniqueness of normal forms. Therefore, equality of two terms \(s\) and \(t\) can be checking syntactic equality of their normal forms \(N(s)\) and \(N(t)\). We begin by introducing a refined notion of equality, which will be essential for proving many of the syntactic properties we need in this section.

\begin{definition}
  Define the \(n\)-bounded equality relation as follows: Let \(\Gamma \vdash s =_n t\) when \(\Gamma \vdash s = t\) with a derivation that only uses rules \((\Delta, s', t') \in \mathcal{R}\) when \(\dim(s') < n\). We further define maximal equality by letting \(\Gamma \vdash \sigma \equiv^{\mathsf{max}} \tau\) when substitutions \(\sigma\) and \(\tau\) are syntactically equal when applied to locally maximal variables, and \(\Gamma \vdash \sigma =^{\mathsf{max}} \tau\) when they are definitionally equal on all locally maximal variables.
\end{definition}

It is clear that bounded equality implies equality. It is also true that maximal equality (of either variety) between valid substitutions implies equality due to conversion. Further, we have that any equal terms of dimension \(n - 1\) are \(n\)-bounded equal and so if \(\Gamma \vdash \sigma \equiv^{\mathsf{max}} \tau\) then it follows that \(\Gamma \vdash \sigma =_{\dim(\sigma)} \tau\). Lastly, if \(\Gamma \vdash \sigma \equiv^{\mathsf{max}} \tau\) and both \(\sigma\) and \(\tau\) are valid in \(\Catt_\emptyset\) then it follows that \(\sigma \equiv \tau\).

\subsection{Reduction for \Cattsua}
\label{sec:reduction}

\noindent
To define a reduction for \Cattsua, we  define a reduction relation on terms, types and substitutions by mutual induction, and write these reductions \(\Gamma \vdash \_ \leadsto \_\) with appropriate pieces of syntax replacing the underscores with the reflexive transitive closure written \(\Gamma \vdash \_ \leadsto^* \_\). The structural rules for the single step reduction are given in \cref{fig:red-rules} to which we add disc removal, endo-coherence removal on non-identities, and insertion where the head term is not an identity or disc and the inserted argument is either a standard composite or an identity. Not allowing head reductions on identities is crucial for termination and the restriction of insertion simplifies the confluence proof and is justified by the following lemma.

\begin{figure}
  \centering
  \begin{mathpar}
    \inferrule*[Right=Cell]{\Delta \vdash A \leadsto B}{\Gamma \vdash \Coh \Delta A \sigma \leadsto \Coh \Delta B \sigma} \and
    \inferrule*[Right=Arg]{\Gamma \vdash \sigma \leadsto \tau}{\Gamma \vdash \Coh \Delta A \sigma = \Coh \Delta A \tau} \and
    \inferrule{\Gamma \vdash s \leadsto s'}{\Gamma \vdash \arr s A t \leadsto \arr {s'} {A} {t}} \and
    \inferrule{\Gamma \vdash A \leadsto A'}{\Gamma \vdash \arr s A t \leadsto \arr {s} {A'} {t}} \and
    \inferrule{\Gamma \vdash t \leadsto t'}{\Gamma \vdash \arr s A t \leadsto \arr {s} {A} {t'}} \and
    \inferrule{\Gamma \vdash \sigma \leadsto \tau}{\Gamma \vdash \langle \sigma, x \mapsto s \rangle \leadsto \langle \tau, x \mapsto s \rangle}\and
    \inferrule{\Gamma \vdash s \leadsto t}{\Gamma \vdash \langle \sigma, x \mapsto s \rangle \leadsto \langle \sigma, x \mapsto t \rangle}
  \end{mathpar}
  \caption{Structural Reduction Rules}
  \label{fig:red-rules}
\end{figure}

\begin{restatable}{lemma}{insertable}
  \label{lem:insertable}
  Let \((S,P,T,\Gamma, \sigma, \tau)\) be an insertion redex. Further suppose that \(a \equiv \Coh S A \sigma\) is not an identity or disc. Then there exists a term \(s\) with:
  \[a \leadsto^* s =_{\dim(a)} \Coh {\insertion S P T} {A \sub {\kappa_{S,P,T}}} {\insertion \sigma P \tau}\]
  even when \(P\sub{\sigma}\) is not a standard composite or identity.
\end{restatable}

\noindent Before giving a proof, we give some definitions and further lemmas.

\begin{definition}
  For tree \(S\) and branch \(P\), let
  \[S \sslash P = \insertion S P {D^{\lh(P)-1}}\]
  and let \(\pi_P = \kappa_{S,P,{D^{\lh(P)-1}}}\).
\end{definition}
\begin{definition}
  Given \(S\) and branch \(P\), if \(2 + \bh(P) \leq \lh(P)\), then there is a branch \(P'\) of \(S \sslash P\), given by the same list as \(P\).
\end{definition}

\begin{lemma}
  \label{lem:pruned-bp}
  If \(S\) has branch \(P\) with \(2 + \bh(P) \leq \lh(P)\), then \(\lfloor P' \rfloor \equiv d_{\lh(P) - 1} \sub {\iota_{S,P,D^{\lh(P) - 1}}} \). If \((S,P,T)\) is an insertion point, we further get that \(\insertion {(S \sslash P)} {P'} T = \insertion S P T\) and \(\pi_P \circ \kappa_{S \sslash P,P',T} =^{\mathsf{max}} \kappa_{S,P,T}\). If we are also given \(\sigma : S \to \Gamma\) and \(\tau : T \to \Gamma\) then:
  \[\insertion {(\insertion {\sigma} P {(\{\mathcal{T}_T^{\lh(P) - 1}, \mathcal{U}_T^{\lh(P) - 1}\} \circ \tau)})} {P'} {\tau} \equiv^{\mathsf{max}} \insertion \sigma P \tau\]
\end{lemma}
\begin{proof}
  See\begin{tabular}[t]{l}
    \small\texttt{Catt/Tree/Insertion/Properties.agda(1094)}\\
    \small\texttt{Catt/Typing/Insertion/Equality.agda(139)}.
  \end{tabular}\\[-9pt]\qedhere
\end{proof}

\begin{lemma}
  \label{lem:insertion-irrel}
  If \(P\) is a branch of \(S\), and \(\sigma, \sigma' : S \to \Gamma\) are substitutions differing only on \(\lfloor P \rfloor\), then the following holds for insertion redex \((S,P,T,\Gamma,\sigma,\tau)\):
  \[\insertion \sigma P \tau \equiv \insertion {\sigma'} P \tau\]
\end{lemma}
\begin{proof}
  By inspection of the definition, \(\insertion \sigma P \tau\) does not use the term \(\lfloor P \rfloor \sub \sigma\).
\end{proof}

\begin{proof}[Proof of \cref{lem:insertable}]
  We proceed by induction on \(\lh(P) - \dim(T)\). If \(\lh(P) - \dim(T) = 0\) then \(\mathcal{C}_T^{\lh(P)}\) is a composite and so we can perform the usual insertion.
  We now assume that \(\lh(P) > \dim(T)\). We may also assume without loss of generality that \(\mathcal{C}_T^{\lh(P)}\) is not an identity, as otherwise it would be immediately insertable. This allows us to perform endo-coherence removal to get:
  \[\mathcal{C}_T^{\lh(P)} \leadsto \mathbbm{1}\sub{\{\mathcal{T}_T^{\lh(P) - 1}, \mathcal{U}_T^{\lh(P) - 1}\} \circ \tau}\]
  Now suppose \(a \equiv \Coh S A \sigma\) and \(b \equiv \Coh S A {\sigma'}\) where \(\sigma'\) is the result of applying the above reduction to the appropriate term of \(\sigma\). Since \(P \sub {\sigma'}\) is now an identity it can be inserted to get \(b \leadsto c\) where:
  \begin{equation*}
    c \equiv \Coh {S \sslash P} {A \sub {\pi_P}} {\insertion {\sigma'} P {(\{\mathcal{T}_T^{\lh(P) - 1}, \mathcal{U}_T^{\lh(P) - 1}\} \circ \tau)}}
  \end{equation*}
  We now wish to show that \(2 + \bh(P) \leq \lh(P)\) so that \(P'\) exists as a branch of \(S \sslash P\). Since we always have \(1 + \bh(P) \leq \lh(P)\), we consider the case where \(1 + \bh(P) = \lh(P)\). We know that \(\bh(P) \leq \dim(T) \leq \lh(P)\) and so must be equal to one of the two. If \(\dim(T) = \lh(P)\) then \(\mathcal{C}_T^{\lh(P)}\) is an standard composite. If \(\dim(T) = \bh(P)\) then \(\th(T) = \dim(T)\) and so \(T\) is linear. However this makes \(\mathcal{C}_T^{\lh(P)}\) an identity. Either case is a contradiction and so \(2 + \bh(P) \leq \lh(P)\) and so \(P'\) is a branch of \(S \sslash P\).

  By \cref{lem:pruned-bp,lem:insertion-subs-comm}, we now have:
  \begin{align*}
    &\phantom{{}\equiv{}}P' \sub {\insertion{\sigma'} P {(\{\mathcal{T}_T^{\lh(P) - 1}, \mathcal{U}_T^{\lh(P) - 1}\} \circ \tau)}} \\
    &\equiv d_{\lh(P) - 1} \sub {\iota_{S,P,D^{\lh(P) - 1}} \circ (\insertion {\sigma'} P {(\{\mathcal{T}_T^{\lh(P) - 1}, \mathcal{U}_T^{\lh(P) - 1}\} \circ \tau)})} \\
    &\equiv d_{\lh(P) - 1} \sub {\{\mathcal{T}_T^{\lh(P) - 1}, \mathcal{U}_T^{\lh(P) - 1}\} \circ \tau} \\
    &\equiv \mathcal{T}_T^{\lh(P) - 1}\sub{\tau}\\
    &\equiv \mathcal{C}_T^{\lh(P) - 1}\sub\tau
  \end{align*}
  with the last equivalence holding as if \(\mathcal{T}_T^{\lh(P)-1}\) was a variable then \(\mathcal{C}_T^{\lh(P)}\) would be an identity. As \(\lh(P') - \dim(T) = \lh(P) - \dim(T) - 1\) we can use the induction hypothesis to get that \(c \leadsto d\) and:
  \begin{align*}
    d =_{\dim(a)}{} &\Coh {\insertion {(S \sslash P)} {P'} T} {A \sub {\pi_P \circ \kappa_{S\sslash P,P',T}}} {\\&\insertion {(\insertion {\sigma'} P {(\{\mathcal{T}_T^{\lh(P) - 1}, \mathcal{U}_T^{\lh(P) - 1}\} \circ \tau)})} {P'} {\tau}}
  \end{align*}
  By \cref{lem:pruned-bp,lem:insertion-irrel},
  \begin{equation*}
    d =_{\dim(a)} \Coh {\insertion S P T} {A \sub {\kappa_{S,P,T}}} {\insertion \sigma P \tau}
  \end{equation*}
  which completes the proof as \(a \leadsto^* d\).
\end{proof}

This allows us to prove that the reduction relation agrees with equality.

\begin{prop}
  The reflexive, symmetric, transitive closure of reduction is equal to definitional equality.
\end{prop}
\begin{proof}
  Reduction is clearly a subrelation of definitional equality. For the other direction we mutually induct on subterms and dimension. Assume we have equality \(\Gamma \vdash s = t\). The only difficult cases are endo-coherence removal and insertion. If \(s\) is an identity and the equality is given by endo-coherence removal then \(s \equiv t\).

  Now suppose \(s = t\) is given by insertion. First assume \(s\) is an identity. By \cref{lem:disc-insertion-1} and some computation we have:
  \[ s \equiv \mathbbm{1}\sub{\{\mathcal{C}_\Delta^n\sub{\tau}, A\}} = \Coh {\Delta} {\mathcal{C}_\Delta^n \to_{\mathcal{U}_{D^n}^n\sub{\kappa}} \mathcal{C}_\Delta^n} {\tau}\equiv t\]
  We now notice that applying endo-coherence removal to \(t\) gives:
  \[t \leadsto \mathbbm{1}\sub{\{\mathcal{C}_\Delta^n\sub{\tau}, \mathcal{U}_{D^n}^n \sub \kappa \sub \tau\}}\]
  By validity we must have \(\mathcal{U}_{D^n}^n \sub \kappa\sub \tau = A\) and so by inductive hypothesis this case is done. Next suppose \(s\) is a disc. Then by \cref{lem:disc-insertion-1}:
  \[ s \equiv \mathcal{C}_{D^n}^n\sub{\{\mathcal{C}_\Delta^n\sub{\tau}, A\}} = \Coh \Delta {\mathcal{U}_{D^n}^n \sub \kappa} {\tau} \equiv t\]
  By \cref{lem:standard-type-exterior}, \(t =_n \mathcal{C}_\Delta^n \sub \tau\), and so by applying disc removal to \(s\) and using the inductive hypothesis on dimension we are done.

  Lastly suppose \(s\) is not an identity or disc and the equality is given by insertion. Then by \cref{lem:insertable} and the inductive hypothesis on dimension the proof is complete.
\end{proof}

Due to the conversion condition, if we start with a well typed term, then any term arising as a reduction of it will also be well typed. In practice this means that we rarely need to check typing conditions when reducing a term, and as such we will omit the context \(\Gamma\) and just write \(s \leadsto t\) or \(s \leadsto^* t\), when we know \(s\) is well-typed. By inspecting the proof that \(\sua\) is tame, in particular noting that the proofs do not use the symmetry of equality, we can deduce that \(\leadsto^*\) respects context extension and substitution.

\mypara{Termination.}
We show strong termination for the reduction, demonstrating that there are no infinite reduction sequences. Our strategy is to assign an ordinal number to each term, show that each single step reduction reduces the associated ordinal number, and therefore deduce that any infinite reduction sequence of the form above would imply the existence of an infinite chain of ordinals, which cannot exist due to well-foundedness of ordinal numbers. We call the ordinal number associated to each term its \emph{syntactic complexity}.

To define syntactic complexity, we will need to use ordinal numbers up to \(\omega^\omega\). We will also need a construction known as the natural sum of ordinals, \(\alpha \+ \beta\), which is associative, commutative, and strictly monotone in both of its arguments~\cite{lipparini16_infin_natur_sum}.

\begin{definition}
  For all terms \(t\) and substitutions \(\sigma\), the \emph{syntactic complexity} \(\sc(t)\) and \(\sc(\sigma)\) are mutually defined as follows:
  \begin{itemize}
  \item For substitutions we have:
    \[\sc(\langle t_0, \dots, t_n \rangle) = \bighash_{i=0}^n t_i\]
  \item For terms, we have \(\sc(x) = 0\) for variables \(x\).

  If \(\Coh \Delta A \sigma\) is an identity then:
    \begin{align*}
    &\sc(\Coh \Delta A \sigma) = \omega^{\dim(A)} \+ \sc(\sigma)\\
      \intertext{Otherwise:}
    &\sc(\Coh \Delta A \sigma) = 2\omega^{\dim(A)} \+ \sc(\sigma)
    \end{align*}
  \end{itemize}
\end{definition}

The motivation for syntactic complexity is as follows. We would like to show that each reduction reduces the depth of the syntax tree, but this doesn't quite work, as reductions like insertion can add new constructions into the reduced term. The necessary insight is that these constructions only add complexity in a lower dimension than the term being reduced. The syntactic complexity is given as an ordinal to leverage known results, though it should be noted that ordinals below \(\omega^\omega\) can be represented by a lists of natural numbers ordered lexicographically and under this interpretation the syntactic complexity effectively computes the number of coherences at each dimension. Therefore removing a coherence of dimension \(n\) reduces the complexity, even if we add arbitrary complexity at lower dimensions. Syntactic complexity also treats identities in a special way, as these play a special role in blocking reduction in the theory.

The syntactic complexity does not account for the type in a coherence, as this is difficult to encode. Instead of showing that all reductions reduce syntactic complexity, we instead show that all reduction which are not ``cell reductions'' (reductions that have the rule marked ``Cell'' in their derivation) reduce syntactic complexity and deduce that a hypothetical infinite reduction sequence must only consist of cell reductions after a finite number of steps, and then appeal to an induction on dimension.

\begin{restatable}{theorem}{screduce}
  \label{thm:sc-reduces}
  One-step reductions that do not use the cell rule reduce syntactic complexity. Those that do use the cell rule do not change the complexity.
\end{restatable}

\noindent We begin the proof of this theorem with the following lemma.

\begin{lemma}
  The following inequality holds for any insertion redex \((S,P,T,\Gamma,\sigma,\tau)\):
  \[\sc(\insertion \sigma P \tau) < \sc(\sigma)\]
\end{lemma}
\begin{proof}
  We begin by noting that:
  \begin{align*}
    \sc(\sigma) &= \left(\bighash_{x\neq \lfloor P \rfloor} \sc(x\sub\sigma)\right) \+ \sc(P\sub\sigma)\\
                &= \left(\bighash_{x\neq \lfloor P \rfloor} \sc(x\sub\sigma)\right) \+ \sc(\mathcal{C}_T^{\lh(P)}\sub\tau)\\
                &> \left(\bighash_{x\neq \lfloor P \rfloor} \sc(x\sub\sigma)\right) \+ \sc(\tau)
  \end{align*}
  Further we extend the notion of syntactic depth to labels in the obvious way and therefore show that for all labels \(L\) and \(M\) with appropriate conditions that:
  \[\sc(\insertion L P M) \leq \bighash_{x\neq \lfloor P \rfloor} \sc(x\sub{L}) \+ \sc(M) \]
  which we do by induction on \(P\). If \(P = [k]\) then it is clear that \(\insertion L P M\) contains all the terms of \(M\) and some of the terms of \(L\), and crucially not \(\lfloor P \rfloor \sub L\). If instead \(P = k :: P'\) then by induction hypothesis we get that:
  \[\sc(\insertion {L_k} {P'} {M_1}) \leq \bighash_{x\neq \lfloor P' \rfloor} \sc(x\sub{L'}) \+ \sc(M_1)\]
  It is then clear again that \(\insertion L P M\) contains terms from \(M\) and terms of \(L\) which are not \(P \sub L\), and so the inequality holds.
\end{proof}

\begin{proof}[Proof of \cref{thm:sc-reduces}]
  We wish to show that for all reductions \(s \leadsto t\), and \(\sigma \leadsto \tau\) that \(\sc(t) < \sc(s)\), and \(\sc(\tau) < \sc(\sigma)\) respectively (replacing these by equalities if the reduction was a cell reduction). We proceed by induction on the derivation of the reduction, noting that all cases for structural rules follow from strict monotonicity of the natural sum. This leaves us with the following base cases.

\mypara{Disc removal.} Suppose \(\mathcal{C}_{D^n}^n\sub{\{A,s\}} \leadsto s\) is by disc removal. By a simple induction, \(\sc(\{A,s\}) \geq \sc(s)\) and so:
    \[\sc(s) \leq \sc(\{A,s\}) < \sc(\mathcal{C}_{D^n}^n\sub{\{A,s\}})\]

\mypara{Endo-coherence removal.} Let \(\Coh \Delta {\arr s A s} \sigma \leadsto \mathbbm{1}_{\{A\sub\sigma, s\sub\sigma\}}\) be a reduction by endo-coherence removal. Then:
    \begin{align*}
      \sc(\mathbbm{1}_{\{A \sub \sigma, s \sub \sigma\}}) &= \omega^{1+\dim(A)} \+ \sc(\{A\sub\sigma,s\sub\sigma\})\\
                                                          &< \omega^{1+\dim(A)} \+ \omega^{1+\dim(A)}\\
                                                          &\leq \sc(\Coh \Delta {\arr s A s} \sigma)
    \end{align*}
    Here the second line holds as \(\dim(s) = 1 + \dim(A)\) and the last line holds as \(\Coh \Delta {\arr s A s} \sigma\) cannot be an identity by assumption.

    \mypara{Insertion.} Let \((S,P,T,\Gamma,\sigma,\tau)\) be an insertion redex so that:
    \[\Coh S A \sigma \leadsto \Coh {\insertion S P T} {A \sub \kappa} {\insertion \sigma P \tau}\]
    by insertion. This implies that \(\Coh S A \sigma\) is not an identity. Then:
    \begin{align*}
      &\hspace{-1cm}\sc(\Coh {\insertion S P T} {A \sub \kappa} {\insertion \sigma P \tau})\\
      \leq{} &2\omega^{\dim(A)} \+ \sc(\insertion \sigma P \tau)\\
      <{} &2\omega^{\dim(A)} \+ \sc(\sigma)\\
      \leq{} &\Coh S A \sigma
    \end{align*}

  A simple induction shows that reductions using the cell rule do not modify the complexity, as when the cell rule is used, we modify a type that does not contribute to the syntactic complexity.
\end{proof}

\begin{cor}
  Reduction for \Cattsua is strongly terminating.
\end{cor}
\begin{proof}
  We proceed by a strong induction on dimension.
  Suppose there is an infinite reduction sequence, starting with a dimension\-\(k\) term:
  \[ s_0 \leadsto s_1 \leadsto s_2 \leadsto \cdots\]
  Then by \cref{thm:sc-reduces}, only finitely many of these reductions do not use the cell rule, and so there is an \(n\) such that:
  \[ s_n \leadsto s_{n+1} \leadsto \cdots\]
  are all cell reductions. Each of these reductions reduces one of finitely many subterms of \(s_n\), and each of these subterms has dimension less than \(k\), so by inductive hypothesis, none of these subterms can be reduced infinitely often, contradicting the existence of an infinite reduction sequence.
\end{proof}

\mypara{Confluence.}
To prove confluence, we take the standard approach of proving local confluence, which says that all single-step reductions of a term can be reduced (in any number of steps) to a common reduct. This implies full confluence (that multi-step reduction has the diamond property) and uniqueness of normal forms when combined with strong termination.%\todo{citation?}.

\begin{restatable}{theorem}{confluence}
  \label{thm:confluence}
  Reduction is locally confluent: if \(a\) is valid with \(a \leadsto b\) and \(a \leadsto c\), then there exists some \(d\) with \(b \leadsto^* d\) and \(c \leadsto^* d\).
\end{restatable}

\begin{proof}
  We proceed by simultaneous induction on subterms and dimension. Suppose \(a \leadsto b\) and \(a \leadsto c\). It is sufficient to show that \(b \leadsto^* b'\), \(c \leadsto^* c'\), and that \(b' =_{\dim(a)} c'\), as then by induction on dimension we have that \(b'\) and \(c'\) have a common reduct, which we can obtain for example by reducing both terms to normal form. Choosing $a \equiv \Coh S A \sigma$, the most nontrivial case is where $a \leadsto b$ is an instance of insertion along some branch \(P\), and \(a \leadsto c\) is an insertion on the argument \(\lfloor P \rfloor\sub\sigma\). The difficulty of this critical pair is that \(\lfloor P \rfloor\sub \sigma\) need not be in head normal form, and furthermore, the reduction \(a \leadsto c\) can make the original insertion invalid. This phenomenon does not occur in the predecessor theory \(\Cattsu\), where only identities can be pruned, and all reducts of identities are again identities. We begin with this critical pair below.

  \mypara{\bf Insertion.} Let \(a \leadsto b\) be an insertion along redex \((S,P,T,\Gamma,\sigma,\tau)\) with \(a \equiv \Coh S A \sigma\) not an identity or disc and  with \(\mathcal{C}_T^{\lh(P)}\) being a composite or identity. We now split on the reduction \(a \leadsto c\).
  \mypara{Insertion on inserted argument.} Now suppose \(\mathcal{C}_T^{\lh(P)}\sub\tau\) admits an insertion along redex \((T, Q, U, \Gamma, \tau, \mu)\). Then:
\[\mathcal{C}_T^{\lh(P)}\sub\tau \leadsto \Coh {\insertion T Q U} {\mathcal{U}_T^{\lh(P)}\sub{\kappa_{T,Q,U}}} {\insertion \tau Q \mu}\]
We then have \(c \equiv \Coh S A {\sigma'}\) where \(\sigma'\) is \(\sigma\) with the reduction above applied. We can conclude that \(\mathcal{C}_T^{\lh(P)}\) must be a composite (i.e. not an identity) as otherwise the second insertion would not be possible. Similarly \(T\) cannot be linear as otherwise \(\mathcal{C}_T^{\lh(P)}\) would be a disc.

We now need the following lemmas, the first of which which is a directed version of \cref{lem:standard-type-exterior} with more conditions.

\begin{restatable}{lemma}{standardtypereduct}
  \label{lem:standard-type-exterior-reduct}
  Let \((S,P,T)\) be an insertion point. Then if \(S\) is not linear or \(n \leq \dim(S)\), \(\mathcal{U}_S^n\sub{\kappa_{S,P,T}} \leadsto^* \mathcal{U}_{\insertion S P T}^n\) and if \(\dim(S) \leq n\) and \(S\) is not linear or \(\dim(S) = n\) then \(\mathcal{T}_S^n \sub{\kappa_{S,P,T}} \leadsto^* \mathcal{T}_{\insertion S P T}^n\).
\end{restatable}

To proof \cref{lem:standard-type-exterior-reduct}, the interaction of insertion with boundary maps must be investigated. This is done by the following lemmas.

\begin{lemma}
  \label{lem:insertion-bd-1}
  Let \(n \in \mathbb{N}\) and suppose \((S,P,T)\) is an insertion point such that:
  \begin{itemize}
  \item \(n < \lh(P)\)
  \item \(n \leq \th(T)\)
  \end{itemize}
  Then \(\partial_n(S) = \partial_n(\insertion S P T)\) and for \(\epsilon \in \{-,+\}\):
  \[ \delta^\epsilon_n(S) \circ \kappa_{S,P,T} \equiv^{\mathsf{max}} \delta_d^\epsilon(\insertion S P T)\]
\end{lemma}
\begin{proof}
  See \small\texttt{Catt/Tree/Insertion/Properties.agda(810)}.
\end{proof}

\begin{definition}
  Let \(n \in \mathbb{N}\), and suppose \(S\) is a tree with branch \(P\) with \(n > \bh(P)\). Then we can define a new branch \(\bound n P\) of \(\bound n S\) given by the same list as \(P\).
\end{definition}

\begin{lemma}
  \label{lem:insertion-bd-2}
  Let \(n \in \mathbb{N}\) and suppose \((S,P,T)\) is an insertion point such that one of the following holds:
  \begin{enumerate}
  \item \(n > \th(T)\) and \(n \leq \lh(P)\)
  \item \(n \geq \lh(P)\)
  \end{enumerate}
  Then \(\insertion {\bound n S} {\bound n P} {\bound n T} = \bound n {\insertion S P T}\) and:
  \[ \inc \epsilon n S \circ \kappa_{S,P,T} \equiv^{\mathsf{max}} \kappa_{\bound n S,\bound n P,\bound n T} \circ \inc \epsilon n {\insertion S P T}\]
  for \(\epsilon \in \{-,+\}\).
\end{lemma}
\begin{proof}
  See \small\texttt{Catt/Tree/Insertion/Properties.agda(938)}.
\end{proof}

\begin{lemma}
  \label{lem:exterior-interior-insertion}
  Let \((S,P,T)\) be an insertion point. Then:
  \[\insertion {\kappa_{S,P,T}} P {\iota_{S,P,T}} \equiv \id\]
\end{lemma}
\begin{proof}
  See \small\texttt{Catt/Tree/Insertion/Properties.agda(233)}.
\end{proof}

\begin{lemma}
  \label{lem:comp-to-tm}
  For all \(n\) and \(S\), \(\mathcal{C}_S^n \leadsto^* \mathcal{T}_S^n\).
\end{lemma}
\begin{proof}
  The only case in which \(\mathcal{C}_S^n \neq \mathcal{T}_S^n\) is when \(S = D^n\), in which case a single disc removal gives the required reduction.
\end{proof}

\noindent We can now give a proof of \cref{lem:standard-type-exterior-reduct}.
\begin{proof}[Proof of \cref{lem:standard-type-exterior-reduct}]
  We proceed by induction on \(n\), starting with the statement for types. If \(n = 0\) then both standard types are \(\star\), so we are done. Otherwise we have:
  \begin{alignat*}{3}
    &\mathcal{U}_S^{1 + n}\sub {\kappa_{S,P,T}} &&{}\equiv{} &&\mathcal{T}_{\bound n S}^n\sub{\inc - n S \circ \kappa_{S,P,T}}\\
    &&&&&\to_{\mathcal{U}_S^n\sub{\kappa_{S,P,T}}}\\
    &&&&&\mathcal{T}_{\bound n S}^n\sub{\inc + n S \circ \kappa_{S,P,T}}
  \intertext{and}
    &\mathcal{U}_{\insertion S P T}^{1 + n} &&{}\equiv{} &&\mathcal{T}_{\bound n {\insertion S P T}}^n \sub{\inc - n {\insertion S P T}} \\
  &&&&&\to_{\mathcal{U}_{\insertion S P T}^n}\\
  &&&&&\mathcal{T}_{\bound n {\insertion S P T}}^n \sub{\inc + n {\insertion S P T}}
  \end{alignat*}
  By inductive hypothesis: \(\mathcal{U}_S^n \sub{\kappa_{S,P,T}} \leadsto^* \mathcal{U}_{\insertion S P T}^n\), and so we need to show that:
  \[\mathcal{T}_{\bound n S}^n\sub{\inc \epsilon n S \circ \kappa_{S,P,T}} \leadsto^* \mathcal{T}_{\bound n {\insertion S P T}}^n \sub{\inc \epsilon n {\insertion S P T}}\]
  We now note that either the conditions for \cref{lem:insertion-bd-1} or \cref{lem:insertion-bd-2} must hold. If conditions for \cref{lem:insertion-bd-1} hold then (as everything is well typed in \(\Catt_\emptyset\)) we get that the required reduction is trivial. Therefore we focus on the second case. Here we get from \cref{lem:insertion-bd-2} that:
  \[\mathcal{T}_{\bound n S}^n\sub{\inc \epsilon n S \circ \kappa_{S,P,T}} \equiv \mathcal{T}_{\bound n S}^n\sub{\kappa_{\bound n S,\bound n P,\bound n T} \circ \inc \epsilon n {\insertion S P T}}\]
  Then we can apply the inductive hypothesis for terms as if \(n \leq \dim(S)\) then \(\dim(\bound n S) = n\) and otherwise \(\bound n S = S\) is not linear, and so we get the required reduction.

  Now we move on to the case for terms. If \(\mathcal{T}_S^n\) is a variable, then we must have that \(S\) is linear \(S = D^n\). We must also have in this case that \(\mathcal{T}_S^n = \lfloor P \rfloor\). Then by \cref{lem:insertion-subs-comm}, \(\mathcal{T}_S^n \sub {\kappa_{S,P,T}} \equiv \mathcal{C}_T^n \sub {\iota_{S,P,T}}\) and then by \cref{lem:disc-insertion-1,lem:comp-to-tm} this reduces to \(\mathcal{T}_{\insertion S P T}^n\) as required. If \(\mathcal{T}_S^n\) is not a variable, then \(\mathcal{T}_S^n \equiv \mathcal{C}_S^n\) and, \(\mathcal{C}_S^n\) cannot be an identity (as either \(S\) is non linear or \(n = \dim(S)\)). By \cref{lem:insertion-subs-comm} and other assumptions we get that \(\mathcal{C}_S^n \sub {\kappa_{S,P,T}}\) admits an insertion along branching point \(P\) and so:
  \begin{alignat*}{2}
    &&&\mathcal{T}_S^n\sub{\kappa_{S,P,T}}\\
    &\equiv{} &&\mathcal{C}_S^n\sub {\kappa_{S,P,T}}\\
    &\leadsto{} &&\Coh {\insertion S P T} {\mathcal{U}_S^n \sub {\kappa_{S,P,T}}} {\insertion {\kappa_{S,P,T}} P {\iota_{S,P,T}}}\\
    &\equiv{} &&\Coh {\insertion S P T} {\mathcal{U}_S^n \sub {\kappa_{S,P,T}}} {\id}\\
    &\leadsto^*{} &&\Coh {\insertion S P T} {\mathcal{U}_{\insertion S P T}^n} {\id}\\
    &\equiv{} &&\mathcal{C}_{\insertion S P T}^n\\
    &\leadsto^*{} &&\mathcal{T}_{\insertion S P T}^n
  \end{alignat*}
  With the second equivalence coming from \cref{lem:exterior-interior-insertion}, the second reduction coming from inductive hypothesis (which is well founded as the proof for types only uses the proof for terms on strictly lower values of \(n\)), and the last reduction coming from \cref{lem:comp-to-tm}.
\end{proof}

\begin{lemma}
  \label{lem:insert-lin-height}
  Let \((S,P,T)\) be an insertion point. Further assume \(S\) is not linear. Then \(\th(\insertion S P T) \geq \th(S)\).
\end{lemma}
\begin{proof}
  See \small\texttt{Catt/Tree/Insertion/Properties.agda(1165)}.
\end{proof}

By this lemma (as \(T\) is not linear), we have \[\mathcal{U}_T^{\lh(P)}\sub{\kappa_{T,Q,U}} \leadsto^* \mathcal{U}_{\insertion T P Q}^{\lh(P)}\]and so \(\mathcal{C}_T^{\lh(P)} \sub \tau \leadsto^* \mathcal{C}_{\insertion T Q U}^{\lh(P)} \sub {\insertion \tau Q \mu}\). Let \(c'\) be the term obtained by applying this further reduction to the appropriate argument. Now by \cref{lem:insert-lin-height}, we have that \(\th(\insertion T Q U) \geq \th(T)\) and so by \cref{lem:insertable}, there is \(c' \leadsto^* c''\) with:
\begin{align*}
  c'' =_{\dim(a)}{} &\Coh {\insertion S P {(\insertion T Q U)}} {\\&A \sub {\kappa_{S,P,\insertion T Q U}}} {\insertion \sigma P {(\insertion \tau Q \mu)}}
\end{align*}
We now examine how \(b\) reduces. We first give a construction of a branch on an inserted context and give some properties of the resulting insertion.

\begin{definition}
  Let \((S,P,T)\) be an insertion point. Further assume \(T\) is not linear and has a branch \(Q\). Then there is a branch \(\insertion S P Q\) of \(\insertion S P T\) with the same height as \(Q\). The new branch points to the same place as \(Q\), except in the copy of \(T\) which exists in \(\insertion S P T\).
\end{definition}

\begin{lemma}
  \label{lem:inserted-insertion}
  Let \((S,P,T)\) be an insertion point. Further assume \(T\) is not linear and has a branch \(Q\). Then \(\lfloor \insertion S P Q \rfloor \equiv Q \sub{\iota_{S,P,T}}\). Further if \((T,Q,U)\) is an insertion point, then
  \begin{alignat*}{3}
    &\insertion S P {(\insertion T Q U)} & &= & &\insertion {(\insertion S P T)} {\insertion S P Q} U\\
    &\kappa_{S,P,\insertion T Q U} &&=^{\mathsf{max}} &&\kappa_{S,P,T} \circ \kappa_{\insertion S P T, \insertion S P Q, U}\\
    &\insertion \sigma P {(\insertion \tau Q \mu)} &&\equiv^{\mathsf{max}} &&\insertion {(\insertion \sigma P \tau)} {\insertion S P Q} \mu
  \end{alignat*}
  for any \(\sigma : S \to \Gamma\), \(\tau : T \to \Gamma\), and \(\mu : U \to \Gamma\).
\end{lemma}
\begin{proof}
  See\begin{tabular}[t]{l}
    \small\texttt{Catt/Tree/Insertion/Properties.agda(1190)}\\
    \small\texttt{Catt/Typing/Insertion/Equality.agda(489)}.
  \end{tabular}\\[-9pt]\qedhere
\end{proof}

\noindent
As \(T\) is not linear, there is a branch \(\insertion S P Q\) of \(\insertion S P T\) and we get the following by \cref{lem:inserted-insertion,lem:insertion-subs-comm}:
\begin{align*}
  \insertion S P Q \sub {\insertion \sigma P \tau}
  &\,\equiv\, Q \sub {\iota_{S,P,T} \circ (\insertion \sigma P \tau)}
  %\\&
  \,\equiv\, Q \sub \tau
  %\\&
  \,\equiv\, \mathcal{C}_U^{\lh(Q)}\sub\mu
\end{align*}
Since \(\th(U) \geq \bh(Q) = \bh(\insertion S P Q)\) we can reduce \(b\) to \(b'\) by insertion as follows:
\begin{align*}
  b' \equiv{} &\Coh {\insertion {(\insertion S P T)} {\insertion S P Q} U} {\\&A \sub {\kappa_{S,P,T} \circ \kappa_{\insertion S P T, \insertion S P Q, U}}} {\insertion {(\insertion \sigma P \tau)} {\insertion S P Q} \mu}
\end{align*}
and then by \cref{lem:inserted-insertion} we get \(b' =_{\dim(a)-1} c''\) as required.

\mypara{Insertion.} Suppose \(a \leadsto c\) is also an insertion, along a branch \(Q\) of \(S\). We now split on whether \(\lfloor P \rfloor = \lfloor Q \rfloor\). First suppose \(\lfloor P \rfloor = \lfloor Q \rfloor\); then by the following lemma, we have \(b =_{\dim(a)} c\).

\begin{lemma}
  \label{lem:insertion-same}
  Suppose \((S,P,T)\) and \((S,Q,T)\) are insertion points with \(\lfloor P \rfloor \equiv \lfloor Q \rfloor\). Then \(\insertion S P T = \insertion S Q T\) and \(\kappa_{S,P,T} \equiv^{\mathsf{max}} \kappa_{S,Q,T}\). If we further have \(\sigma : S \to \Gamma\) and \(\tau : T \to \Gamma\), then \(\insertion \sigma P \tau \equiv^{\mathsf{max}} \insertion \sigma Q \tau\).
\end{lemma}
\begin{proof}
  See \small\texttt{Catt/Tree/Insertion/Properties.agda(290)}.
\end{proof}

Suppose now that \(\lfloor P \rfloor \neq \lfloor Q \rfloor\), and that \(Q \sub \sigma \equiv \mathcal{C}_U^{\lh(Q)} \sub \mu\), such that
\[c \equiv \Coh {\insertion S Q U} {A \sub {\kappa_{S,Q,U}}} {\insertion \sigma Q \mu}\]
For this case we define the following branching point of the inserted tree.

\begin{definition}
  Let \((S,P,T)\) be an insertion point and \(Q\) be a branch of \(S\) such that \(\lfloor P \rfloor \neq \lfloor Q \rfloor\). Then we can define a new branch \(\insertion Q P T\) of \(\insertion S P T\) with \(\bh(\insertion Q P T) = \bh(Q)\) and \(\lfloor \insertion Q P T \rfloor \equiv Q \sub {\kappa_{S,P,T}}\). Intuitively this branch refers to the same part of \(S\), and is unaffected by \(T\) being inserted in.
\end{definition}

We now want \(b\) and \(c\) to further reduce as follows:
      \begin{align*}
        b &\leadsto b' \equiv \Coh {\insertion {(\insertion S P T)} {\insertion Q P T} U} {\\&A\sub{\kappa_{S,P,T} \circ \kappa_{\insertion S P T, \insertion Q P T, U}}} {\insertion {(\insertion \sigma P \tau)} {\insertion Q P T} \mu}\\
        c &\leadsto c' \equiv \Coh {\insertion {(\insertion S Q U)} {\insertion P Q U} T} {\\&A\sub{\kappa_{S,Q,U} \circ \kappa_{\insertion S Q U, \insertion P Q U, T}}} {\insertion {(\insertion \sigma Q \mu)} {\insertion P Q U} \tau}
      \end{align*}
      We will show that the first reduction is valid and the other will hold by symmetry. We need \(b\) to not be an identity. First suppose that \(a\) is typed by the coherence rule. Then if \(A \equiv \arr s B t\) we must have \(\supp(s) = \supp(S)\). However \(S\) is certainly not linear, as it has two distinct leaves. Therefore \(s\) cannot be a variable and so \(s \sub {\kappa_{S,P,T}}\) is also not a variable, making \(b\) not an identity. Now suppose instead that \(a\) is a composite. Then, if \(\partial^-(S)\) is not linear, we can apply the same logic as before, and so \(S\) must have linear height one less than its dimension. However, this means that both \(\lfloor P \rfloor\) and \(\lfloor Q \rfloor\) are distinct variables with the same dimension as the tree, and so \(\dim(\insertion S P T) = \dim(S) = \dim(a)\) and so \(b\) cannot be an identity.

      We also note:
      \begin{align*}
        \insertion Q P T \sub {\insertion \sigma P \tau} &\equiv Q \sub{\kappa} \sub{\insertion \sigma P \tau}\\
                                                                         &\equiv Q \sub{\sigma}\\
                                                                         &\equiv \mathcal{C}_U^{\lh(Q)}\sub\mu\\
                                                                         &\equiv \mathcal{C}_U^{\lh(\insertion Q P T)}\sub\mu
      \end{align*}
      as required for the insertion, with the third equality coming from the following lemma.

      \begin{lemma}
        \label{lem:ins-comm-max}
        For insertion redex \((S,P,T,\Gamma,\sigma,\tau)\), the following hold:
        \[ \iota_{S,P,T} \circ (\insertion \sigma P \tau) \equiv \tau\]
        \[ \kappa_{S,P,T} \circ (\insertion \sigma P \tau) \equiv^{\mathsf{max}} \sigma \]
        These imply the equality results from \cref{sec:insert-props}.
      \end{lemma}
      \begin{proof}
        See \small\texttt{Catt/Tree/Insertion/Properties.agda(108)}.
      \end{proof}
      Lastly the trunk height condition is satisfied as \(\bh(Q) = \bh(\insertion Q P T)\), and so both reductions are valid insertions. We now give the following definition and lemmas, from which it will follow that \(b' =_{\dim(a)} c'\).

      \begin{definition}
  We define a variant of the inserted substitution, and write it \(\insertionprime \sigma P \tau\). Whereas the original uses as many terms from \(\tau\) as possible, the variant uses as many terms from \(\sigma\) as possible. More precisely, we define \(\insertionprime L {[k]} M\) to be:
    \[\vcenter{\hbox{\raisebox{-5pt}{\(s_0\)}{\(L_1\)}\raisebox{-5pt}{\(s_1\)}}}\ \cdots\ \vcenter{\hbox{\(L_{k-1}\)\raisebox{-5pt}{\(\bm{s_k}\)}\(M_1\)\raisebox{-5pt}{\(t_2\)}}}\ \cdots\  \vcenter{\hbox{\(M_m\)\raisebox{-5pt}{\(\bm{s_{k+1}}\)}\(L_{k+1}\)\raisebox{-5pt}{\(s_{k+1}\)}}}\ \cdots\ \vcenter{\hbox{\(L_n\)\raisebox{-5pt}{\(s_n\)}}}\]
    and \(\insertionprime L {k :: P'} M\) as:
    \[\vcenter{\hbox{\raisebox{-5pt}{\(s_0\)}{\(L_1\)}\raisebox{-5pt}{\(s_1\)}}}\ \cdots\ \vcenter{\hbox{{\(L_{k-1}\)}\raisebox{-5pt}{\(\bm{s_k}\)}\((\insertion {L_k} {P'} {M_1})\)\raisebox{-5pt}{\(\bm{s_{k+1}}\)}\(L_{k+1}\)\raisebox{-5pt}{\(s_{k+1}\)}}}\ \cdots\ \vcenter{\hbox{\(L_n\)\raisebox{-5pt}{\(s_n\)}}}\]
    where the terms in bold have been modified from the original definition. In the edge case where \(M = \verb|[]|\), we arbitrarily use \(s_k\) instead of \(s_{k+1}\) for the definition of \(\insertionprime L {[k]} M\).
\end{definition}

\begin{lemma}
  \label{lem:insertionprime}
  The following equality holds
  \[\insertion \sigma P \tau =_{\dim(S)} \insertionprime \sigma P \tau\]
  for any insertion redex \((S,P,T,\Gamma,\sigma,\tau)\).
\end{lemma}
\begin{proof}
  See \small\texttt{Catt/Tree/Insertion/Typing.agda(188)}.
\end{proof}

\begin{lemma}
  \label{lem:insertion-different}
  Let \((S,P,T)\) and \((S,Q,U)\) be insertion points such that \(\lfloor P \rfloor \neq \lfloor Q \rfloor\). Then we have:
  \[ \insertion {(\insertion S P T)} {\insertion Q P T} U = \insertion {(\insertion S Q U)} {\insertion P Q U} T\]
  \[ \kappa_{S,P,T} \circ \kappa_{\insertion S P T, \insertion Q P T, U} \equiv^{\mathsf{max}} \kappa_{S,Q,U} \circ \kappa_{\insertion S Q U, \insertion P Q U, T}\]
Further
\[ \insertionprime {(\insertion \sigma P \tau)} {\insertion Q P T} \mu \equiv^{\mathsf{max}} \insertionprime {(\insertion \sigma Q \mu)} {\insertion P Q U} \tau\]
for any insertion redexes \((S,P,T,\Gamma,\sigma,\tau)\) and \((S,P,T,\Gamma,\sigma,\mu)\).
\end{lemma}
\begin{proof}
  See \small\texttt{Catt/Tree/Insertion/Properties.agda(662)}.
\end{proof}

\mypara{Cell reduction.} If \(A \leadsto B\) and \(c \equiv \Coh S B \sigma\) from cell reduction, then if \(c\) is not an identity or disc then it admits insertion to reduce to:
      \[c' \equiv \Coh {\insertion S P T} {B \sub {\kappa_{S,P,T}}} {\insertion \sigma P \tau}\]
      As reduction is compatible with substitution, \(b\) also reduces to~\(c'\). If instead \(c\) was an identity then
      \begin{align*}
        b &\equiv \Coh {\insertion {D^n} P T} {A \sub {\kappa_{S,P,T}}} {\insertion \sigma P \tau}\\
          &\leadsto \Coh {\insertion {D^n} P T} {\mathcal{U}_{D^n}^{n+1}\sub {\kappa_{S,P,T}}} {\insertion \sigma P \tau}\\
          &\leadsto^* \mathbbm{1}\sub{\{\mathcal{U}_{D^n}^n, d_n\} \circ \kappa_{S,P,T} \circ \insertion \sigma P \tau}\\
          &=_{n+1} \mathbbm{1}\sub{\sigma}\\
          &\equiv c
      \end{align*}
      Where the equality is due to \cref{lem:insertion-subs-comm} and \(\{\mathcal{U}_{D^n}^n, d_n\}\) being the identity substitution, and the second reduction is due to the following lemma.

\begin{lemma}
  \label{lem:always-ecr}
  The following reduction holds, even when the left-hand side is an identity:
  \[\Coh \Gamma {\arr s A s} \sigma \leadsto^* \mathbbm{1}\sub{\{s,A\} \circ \sigma}\]
\end{lemma}
\begin{proof}
  If \(\Coh \Gamma {\arr s A s} \sigma\) is not an identity then we can reduce by endo-coherence removal. Otherwise we have \(\Gamma = D^n\) for some \(n\), \(s \equiv d_n\), and \(A \equiv \mathcal{U}_{D^n}^n\), and so:
  \[\mathbbm{1}\sub{\{s,A\} \circ \sigma} \equiv \mathbbm{1}\sub{\{d_n,\mathcal{U}_{D^n}^n\} \circ \sigma} \equiv \mathbbm{1} \sub{\sigma}\]
It follows that the reduction is trivial.
\end{proof}

      If \(c\) is a disc then:
      \begin{align*}
        b &\equiv \Coh {\insertion {D^n} P T} {A \sub {\kappa_{S,P,T}}} {\insertion \sigma P \tau}\\
          &\leadsto \Coh {\insertion {D^n} P T} {\mathcal{U}_{D^n}^{n}\sub {\kappa_{S,P,T}}} {\insertion \sigma P \tau}\\
          &=_{n} \Coh T {\mathcal{U}_T^n} {\tau}\\
          &\equiv c
      \end{align*}
      Where the equality is given by \cref{lem:disc-insertion-1,lem:standard-type-exterior}.

\mypara{Disc removal.} By assumption, insertion cannot be applied to discs so this case is vacuous.

\mypara{Endo-coherence removal.} Suppose \(A \equiv \arr s B s\) and \(a \leadsto c\) by endo-coherence removal. In this case \(c \equiv \mathbbm{1}\sub{\{A,s\} \circ \sigma}\) and
      \[ b \equiv \Coh {\insertion S P T} {(\arr s B s) \sub {\kappa_{S,P,T}}} {\insertion \sigma P \tau}\]
      which reduces by endo-coherence removal to:
      \[b' \equiv \mathbbm{1} \sub {\{s,A\} \circ \kappa_{S,P,T} \circ (\insertion \sigma P \tau)}\]
      Finally, by \cref{lem:insertion-subs-comm}, we have that \(\kappa_{S,P,T} \circ (\insertion \sigma P \tau) =_{\dim(S)} \sigma\) and so \(b' = _{\dim(S)} c\) and since \(\dim(S) \leq \dim(a)\), we get \(b' =_{\dim(a)} c\) as required.

\mypara{Reduction of non-inserted argument.} Suppose \(\sigma \leadsto \sigma'\) along an argument which is not \(\lfloor P \rfloor\) and \(c \equiv \Coh S A {\sigma'}\). Then as \(P \sub {\sigma'} \equiv \mathcal{C}_T^{\lh(P)}\), an insertion can still be performed on \(c\) to get:
      \[ c \leadsto c' \equiv \Coh {\insertion S P T} {A \sub {\kappa_{S,P,T}}} {\insertion {\sigma'} P \tau}\]
      By the following lemma, \(b \leadsto^* c'\).

      \begin{lemma}
  \label{lem:from-insertion-rto}
  Given \(\sigma \leadsto^* \sigma'\) and \(\tau \leadsto^* \tau'\), then if \(\insertion \sigma P \tau\) is defined, we have:
  \[ \insertion \sigma P \tau \leadsto^* \insertion {\sigma'} P {\tau'}\]
\end{lemma}
\begin{proof}
  Each term in \(\insertion \sigma P \tau\) is a term of \(\sigma\) or \(\tau\), and so we can simply apply the given reductions.
\end{proof}

\mypara{Argument reduction on inserted argument.} Suppose \(\tau \leadsto \tau'\), and \(\sigma'\) is \(\sigma\) but with the argument for \(\lfloor P \rfloor\) replaced by \(\Coh T {\mathcal{U}_T^{\lh(P)}} {\tau'}\), such that \(\sigma \leadsto \sigma'\) and \(a \leadsto c \equiv \Coh S A {\sigma'}\). Then \(c\) admits an insertion and reduces as follows:
      \[c \leadsto c' \equiv \Coh {\insertion S P T} {A \sub {\kappa_{S,P,T}}} {\insertion {\sigma'} P {\tau'}}\]
      By \cref{lem:from-insertion-rto} we then have \(b \leadsto^* c'\).

\mypara{Disc removal on inserted argument.} If \(a \leadsto c\) is the result of reducing \(P \sub \sigma\) by disc removal, then \(T\) must equal \(D^{n}\) (with \(n = \lh(P)\)) and \(c \equiv \Coh S A {\sigma'}\) where \(\sigma'\) is \(\sigma\) with the argument for \(\lfloor P \rfloor\) replaced with \(d_n\sub\tau\). Further:
      \begin{equation*}
        b \equiv \Coh {\insertion S P {D^n}} {A \sub {\kappa_{S,P,D^n}}} {\insertion \sigma P \tau}
      \end{equation*}
      \noindent We now give the following lemma.
      \begin{lemma}
  \label{lem:disc-insertion-2}
  Let \(S\) be a tree, and \(P\) a branch of \(S\). Then we get that \(\insertion S P {D^{\lh(P)}} = S\) and \(\kappa_{S,P,D^{\lh(P)}} =^{\mathsf{max}} \id\). Further
  \[\insertion \sigma P \tau \equiv^{\mathsf{max}} \sigma\]
  if \((S,P,D^{\lh(P)},\Gamma,\sigma,\tau)\) is an insertion redex.
\end{lemma}
\begin{proof}
  See\begin{tabular}[t]{l}
    \small\texttt{Catt/Tree/Insertion/Properties.agda(184)}\\
    \small\texttt{Catt/Typing/Insertion/Equality.agda(265)}.\\
  \end{tabular}\\[-9pt]\qedhere
\end{proof}
      By this lemma, \(\insertion S P {D^n} = S\), \(\insertion {\sigma'} P \tau \equiv^{\mathsf{max}} {\sigma'}\) and \(\kappa_{S,P,{D^n}} = \id\). Further, by \cref{lem:insertion-irrel}, we have \(\insertion \sigma P \tau \equiv \insertion {\sigma'} P \tau\). As \(\dim(A\sub{\kappa_{S,P,{D^d}}}) \leq \dim(a)\):
      \begin{align*}
        b &\equiv \Coh S {A \sub {\kappa_{S,P,D^d}}} {\insertion {\sigma'} P \tau}\\*
          &=_{\dim(a)} \Coh S A {\sigma'}\\*
          &\equiv c
      \end{align*}

\mypara{Endo-coherence removal on inserted argument.} The inserted argument must already be an standard composite or identity, so cannot reduce by endo-coherence removal, hence this case is vacuous.

\mypara{Insertion on inserted argument.} This case is given in the main body of the paper.

\mypara{\bf Cell reduction.} If \(a \leadsto b\) is an instance of a cell reduction, then \(a \equiv \Coh \Gamma A \sigma\), \(A \leadsto B\), and \(b \equiv \Coh \Gamma B \sigma\). We now split on the reduction \(a \leadsto c\).

\mypara{Cell reduction.}If \(a \leadsto c\) is the result of another cell reduction \(A \leadsto C\) then either the reductions target different parts of the type in which case there is a common reduct \(D\). Otherwise we can appeal to the inductive hypothesis on subterms to find a common reduct.

\mypara{Disc removal.} \(A\) is not in normal form, so \(a\) cannot have a disc as its head, hence \(a \leadsto c\) cannot be a disc removal.

\mypara{Endo-coherence removal.} Suppose \(A \equiv \arr s {A'} s\), and \(c \equiv \mathbbm{1}\sub{\{A',s\} \circ \sigma}\). If the reduction \(A \leadsto B\) arises from \(A' \leadsto B'\) then \(b\) immediately admits endo-coherence removal and reduces as follows:
      \[b' \equiv \mathbbm{1} \sub {\{B',s\} \circ \sigma}\]
      Then \(\{s,A'\} \leadsto \{s, B'\}\) and so \(c \leadsto^* b'\).

      Otherwise we either have the reduction \(\arr s {A'} s \leadsto \arr {s'} {A'} s\) or \(\arr s {A'} s \leadsto \arr s {A'} {s'}\). In either case we have \(b \leadsto \Coh \Gamma {\arr {s'} {A'} {s'}} \sigma\) and so by endo-coherence removal we get the following:
      \[b \leadsto^* b' \equiv \mathbbm{1} \sub{\{A',s'\} \circ \sigma}\]
      Hence we conclude \(c \leadsto^* b'\) as required.

\mypara{Argument reduction.} If \(c \equiv \Coh \Gamma A {\sigma'}\) arises from argument reduction \(\sigma \leadsto \sigma'\) then both \(b\) and \(c\) reduce to \(\Coh \Gamma B {\sigma'}\) by an argument reduction or cell reduction.
    % \end{itemize}

\mypara{\bf Disc removal.} If \(a \leadsto b\) is a disc removal then \(a \equiv \mathcal{C}_{D^n}^n\sub \sigma\) for some \(n\) and \(b \equiv d_n \sub \sigma\). Now \(a\) can not reduce by endo-coherence removal and disc removal is unique so the only remaining case is an argument reduction. Suppose \(\sigma \leadsto \sigma'\) and \(c \equiv \mathcal{C}_{D^n}^n \sub {\sigma'}\), which reduces to \(c' \equiv d_n \sub \sigma'\) by disc removal. If the reduction \(\sigma\leadsto \sigma'\) is along \(d_n\) then by definition \(c \leadsto c'\), and otherwise \(c \equiv c'\).

\mypara{\bf Endo-coherence removal.} For \(a \leadsto b\) to be an instance of Endo-coherence removal we must have \(a \equiv \Coh \Gamma {\arr s A s} \sigma\) and \(b \equiv \mathbbm{1}\sub{\{A,s\}\circ \sigma}\). The only case remaining for the second reduction is argument reduction, as the rest follow by symmetry or the uniqueness of endo-coherence removal. Therefore let \(\sigma \leadsto \sigma'\) and \(c \equiv \Coh \Gamma {\arr s A s} {\sigma'}\) which reduces to \(c' \equiv \mathbbm{1}\sub{\{A,s\} \circ \sigma'}\) by endo-coherence removal. As the transitive closure of reduction respects substitution, \(\{A,s\} \circ \sigma \leadsto^* \{A,s\} \circ \sigma'\) and so \(b \leadsto^* c'\).

\mypara{\bf Argument reduction.} We suppose \(a \equiv \Coh \Gamma A \sigma\) and \(b \equiv \Coh \Gamma A {\sigma'}\) where \(\sigma'\) is the result of reducing one argument \(x\) of \(\sigma\). The only case left is that \(a \leadsto c\) is also an argument insertion, and so \(c \equiv \Coh \Gamma A {\sigma''}\) with \(\sigma''\) the result of reducing an argument \(y\) of \(\sigma\). If \(x\) does not equal \(y\), then both \(\sigma'\) and \(\sigma''\) reduce to the substitution where we apply both reductions. Otherwise if \(x = y\) then by induction on subterms, \(\sigma'\) and \(\sigma''\) have a common reduct.
\end{proof}
\section{Implementation}
\label{sec:implementation}

\noindent
We have provided an OCaml implementation of our type theory \Cattsua. Here we show some example use-cases, indicating the file name of the example in the supplementary material.

\mypara{Triangle Equation.}{\quad\small\texttt{examples/monoidal.catt}}

\noindent
In a monoidal category, the triangle equation expresses compatibility of the left unitor, right unitor and associator:
$$
\begin{tikzpicture}
\node (1) at (0,0) {$(f \cdot \id) \cdot g$};
\node (2) at (3,0) {$f \cdot (\id \cdot g)$};
\node (3) at (1.5,-1) {$f \cdot g$};
\draw [->] (1) to node [above] {$\alpha_{f,\id,g}$} (2);
\draw [->] (2) to node [right=8pt, pos=.8] {$f \cdot \lambda_g$} (3);
\draw [->] (1) to node [left=8pt, pos=.8] {$\rho_f \cdot g$} (3);
\node [rotate=-90] at (1.5,-.4) {$\Rightarrow$};
\end{tikzpicture}
$$
\noindent
We express this in our implementation as a 3-morphism constructed as follows:

\noindent{\quad\footnotesize\texttt{coh triangle (x(f)y(g)z) :}}

\vspace{-1pt}
\noindent{\qquad\footnotesize{\color{white}}\texttt{vert (assoc f (id y) g) (horiz (id1 f) (unitor-l g))}}

\vspace{-1pt}
\noindent{\qquad\footnotesize\texttt{ => horiz (unitor-r f) (id1 g)}}

\noindent
This defines ``triangle'' to be function taking a set of arguments to the coherence with the given type.

\vspace{3pt}\noindent
Since the triangle equation is entirely expressed in terms of associator and unitor structure, we would expect this 3-morphism to normalize to an identity in our strictly associative and unital setting. We can test this as follows:

\noindent{\quad\footnotesize\texttt{normalize \{x :: *\} \{y :: *\} (f :: x => y) \{z :: *\} (g :: y => z)}}

\vspace{-1pt}
\noindent{\qquad\footnotesize\texttt{| triangle f g}}

\vspace{3pt}\noindent
The result is an identity as expected.

\mypara{Pentagon Equation.}{\quad\small\texttt{examples/monoidal.catt}}

\noindent
We next consider the pentagon constraint, the second axiom family of a monoidal category:

\vspace{-12pt}
$$
\begin{tikzpicture}[yscale=-1]
\node (1) at (0,0) {$((f \cdot g) \cdot h) \cdot i$};
\node (2) at (1,1) {$(f \cdot (g \cdot h)) \cdot i$};
\node (3) at (5,1) {$f \cdot ((g \cdot h) \cdot i)$};
\node (4) at (6,0) {$f \cdot (g \cdot (h \cdot i))$};
\node (5) at (3,-1) {$(f \cdot g) \cdot (h \cdot i)$};
\draw [->] (1) to node [left=5pt] {$\alpha_{f,g,h} \cdot i$} (2);
\draw [->] (2) to node [below] {$\alpha_{f,g \cdot h, i}$} (3);
\draw [->] (3) to node [right=5pt] {$f \cdot \alpha_{g, h, i}$} (4);
\draw [->] (1) to node [above left] {$\alpha_{f \cdot g, h, i}$} (5);
\draw [->] (5) to node [above right] {$\alpha_{f, g, h \cdot i}$} (4);
\node [rotate=-90] at (3,0) {$\Rightarrow$};
\end{tikzpicture}
$$

\vspace{-10pt}
\noindent
As for the triangle equation, we can construct this in \Catt as the following 3-morphism:

\noindent{\quad\footnotesize\texttt{coh pentagon (v(f)w(g)x(h)y(i)z) :}}%

\vspace{-1pt}
\noindent{\qquad\footnotesize\hspace{0pt}\texttt{vert (assoc (comp f g) h i) (assoc f g (comp h i))}}%

\vspace{-1pt}
\noindent{\qquad\footnotesize\texttt{=> vert}}

\vspace{-1pt}
\noindent{\qquad\footnotesize\texttt{(horiz (assoc f g h) (id1 i))}}

\vspace{-1pt}
\noindent{\qquad\footnotesize\hspace{00pt}\texttt{(vert (assoc f (comp g h) i)
                                                 (horiz (id1 f) (assoc g h i)))}}

\vspace{3pt}\noindent
Again employing the normalize command, we show that it reduces to the identity as expected.

\mypara{Syllepsis.}{\quad\small\texttt{examples/syllepsis.catt}}

\noindent
The syllepsis is a 5-dimensional homotopy which expresses the fact that the overcrossing and undercrossing are equivalent in 4\-dimensional space:
$$
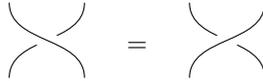
\begin{aligned}
\begin{tikzpicture}
\draw (0,0) to [out=up, in=down] (1,1);
\draw [double=black, white, line width=3pt, double distance=.4pt] (1,0) to [out=up, in=down] (0,1);
\end{tikzpicture}
\end{aligned}
\quad=\quad
\begin{aligned}
\begin{tikzpicture}[xscale=-1]
\draw (0,0) to [out=up, in=down] (1,1);
\draw [double=black, white, line width=3pt, double distance=.4pt] (1,0) to [out=up, in=down] (0,1);
\end{tikzpicture}
\end{aligned}
$$
It is a fundamental move from low-dimensional topology, and plays an essential role in the homotopy groups of spheres. The bureaucracy of weak higher structures means that it has long been recognized as difficult to describe directly in a formal way, given the extensive use of interchangers, unitors and associators that are required to build it.

Two formal models for the syllepsis were presented at LICS 2022, one using homotopy type theory~\cite{sojakova2022syllepsis}, and an alternative using the type theory $\Cattsu$~\cite{StrictUnits}. The theory \Cattsua allows an even shorter representation of the syllepsis, purely in terms of interchanger coherences. The \Cattsua formalisation of the syllepsis makes use of the \texttt{ucomp} builtin throughout, which produces unbiased composites over a pasting diagram specified by a sequence of alternating dimensions and codimensions. As an example ``\texttt{ucomp [ 1 0 3 ]}'' is shorthand for ``the unbiased composite of a 1-cell with a 3-cell along their common 0-cell boundary''. This builtin avoids the boilerplate of defining many composition operations, but adds no expressive power to \Catt, as these can all be defined with the coherence constructor.

% \bibliographystyle{IEEEtran}
% \bibliography{references}
\printbibliography

\end{document}